\documentclass[11pt]{amsart}
\usepackage{graphicx}
\usepackage{amssymb}
\usepackage{epstopdf}
\newcommand{\bbb}[1]{\indent$\bullet$ #1\\}
 
\usepackage{booktabs}
\usepackage{amssymb}
\usepackage[applemac]{inputenc}
\usepackage[english]{babel}
\makeindex
\usepackage{flafter}
\makeatletter
\usepackage{microtype}
\SetTracking[no ligatures={f}]{encoding=*}{40}
\usepackage{graphicx, array}

\usepackage{color, colortbl}
\definecolor{LightCyan}{rgb}{0.88,1,1}
\definecolor{Red}{rgb}{1,0,0}
\definecolor{Green}{rgb}{0,1,0}
\definecolor{Lgray}{rgb}{0.75, 0.75, 0.75}

\usepackage{hyperref}
\hypersetup{
colorlinks=true,       
    linkcolor=blue,          
    citecolor=blue,        
    filecolor=blue,      
    urlcolor=blue           
}

\usepackage[all]{xy}
\usepackage{tikz}
\usepackage{todonotes}
\usepackage[alphabetic,backrefs,lite]{amsrefs}
\usepackage{enumerate}
\usepackage{verbatim}
\usepackage{multirow}
\usepackage{pgf, tikz}
\usepackage{float}
\usepackage{rotating}
\usetikzlibrary{arrows}

\CompileMatrices
\newtheorem{theorem}{Theorem}[section]
\newtheorem{lemma}[theorem]{Lemma}

\newtheorem{proposition}[theorem]{Proposition}
\newtheorem{corollary}[theorem]{Corollary}
\theoremstyle{definition}

\newtheorem{example}[theorem]{Example}

\newtheorem{remark}[theorem]{Remark}
\newtheorem{notation}[theorem]{Notation}

\usepackage{arydshln}

\usepackage{amsmath}

\def\C{{\mathbb C}}

\def\Z{{\mathbb Z}}
\def\R{{\mathbb R}}
\def\N{{\mathbb N}}

\def\P{{\mathbb P}}
\def\F{{\mathbb F}}
\def\Osh{{\mathcal O}}

\def\Pic{\operatorname{Pic}}
\def\Fix{\operatorname{Fix}}

\def\rk{\mbox{rank}}

\def\tr{\operatorname{tr}}

\usepackage[symbol]{footmisc}

\begin{document}

\title[Automorphisms of K3's with one-dimensional moduli space]
{Non-symplectic automorphisms of K3 surfaces with one-dimensional moduli space}

\author{Michela Artebani}
\address{
Departamento de Matem\'atica, \newline
Universidad de Concepci\'on, \newline
Casilla 160-C,
Concepci\'on, Chile}
\email{martebani@udec.cl}

\author{Paola Comparin}
\address{
Departamento de Matem\'atica y Estad\'istica, \newline
Universidad de la Frontera, \newline
Temuco, Chile}
\email{paola.comparin@ufrontera.cl}

\author{Mar\'ia Elisa Vald\'es}
\address{
Departamento de Matem\'atica, \newline
Universidad de Concepci\'on, \newline
Casilla 160-C,
Concepci\'on, Chile}
\email{mariaevaldes@udec.cl}

\subjclass[2000]{Primary 14J28; Secondary 14J50, 14J10}
\keywords{K3 surfaces, automorphisms} 
\thanks{The first and last author have been partially 
supported by Proyecto FONDECYT Regular 
N. 1160897 and N. 1211708, the second author has been partially supported by  
Proyecto FONDECYT Iniciaci\'on en Investigaci\'on N. 11190428,
all authors have been supported by  
Proyecto Anillo ACT 1415 PIA CONICYT}

\begin{abstract}
The moduli space of K3 surfaces $X$ with a purely non-symplectic 
automorphism $\sigma$ of order $n\geq 2$ is one dimensional exactly when $\varphi(n)=8$ or $10$. 
In this paper we classify and give explicit equations for 
the very general members $(X,\sigma)$ of the irreducible components of maximal dimension
of such moduli spaces. In particular we show that there is a unique one-dimensional component for $n=20,22, 24$,
three irreducible components for $n=15$ and  two components in the remaining cases. 
\end{abstract}

\maketitle
\tableofcontents

\section{Introduction} 
An automorphism $\sigma$ of finite order $n\geq 2$ of a complex K3 surface 
$X$ is purely non-symplectic if $\sigma^*(\omega_X)=\zeta_n\omega_X$, where
$\omega_X$ is a nowhere vanishing holomorphic $2$-form of $X$ and
$\zeta_n$ is a primitive $n$th root of unity. 
By \cite[Main Theorem 3]{MO} there exists one such pair $(X,\sigma)$ 
if and only if $n$ belongs to the set $TV_{K3}=\{n\in \N-\{60\}| \varphi(n)\leq 20\}$.

The structure of the moduli space of such K3 surfaces 
can be described by means of the Global Torelli theorem  
and the surjectivity theorem for periods of K3 surfaces (see \cite[\S 11]{DK}). 
In particular it is known that an irreducible component of 
the moduli space of pairs $(X,\sigma)$ for $n\geq 3$ 
is an arithmetic quotient of a Zariski open subset of 
a complex ball  of dimension $\dim(V^{\sigma})-1$,
where $V^{\sigma}$ is the $\zeta_n$-eigenspace of $\sigma^*$ in $H^2(X,\C)$.

In this paper we consider the orders $n$ such that the moduli space of 
 K3 surfaces carrying a purely non-symplectic automorphism of order $n$ is one dimensional.
 We show that the orders $n$ with such property, as expected, are exactly those $n\in TV_{K3}$
 with $\varphi(n)=8$ or $10$, i.e. $11, 15, 16, 20, 22, 24$ and $30$ (see \cite{MO}). 
For all these values of $n$ we classify pairs $(X,\sigma)$ such that $\dim(V^\sigma)=2$, i.e. 
we identify the fixed locus of $\sigma$ and of its powers, determine
the dimensions of the eigenspaces of $\sigma^*$ in $H^2(X,\C)$ 
and compute the N\'eron-Severi lattice of a very general pair.
 The orders $n=11$ and $n=16$ had been previously studied in \cite{OZ2, ArtebaniSartiTaki} 
and \cite{AlTabbaaSartiTaki} respectively.
We collect these results in the following Theorem.

\begin{theorem}\label{table2}
 Let $X$ be a complex K3 surface with a purely non-symplectic automorphism $\sigma$  
of order $n\geq 2$ such that $\varphi(n)=8$ or $10$ and $\dim(V^{\sigma})=2$. 
Then Table \ref{resumen} provides all possible values for the vector $d$ describing the dimensions of the eigenspaces 
of $\sigma^*$ in $H^2(X,\C)$, the topological invariants describing the fixed locus of powers of $\sigma$ (see Section \ref{background}
for the notation) and the N\'eron-Severi lattice of a very general K3 surface in each case. Moreover, all cases in the table exist.
\end{theorem}

 {\scriptsize
 \begin{table}[h!!]
\centering
\begin{tabular}{c|c|c|c|ccc|c}
&$n$&$d$&$i$&$g_i$&$k_i$&$N_i$ & {\rm NS}\\
\hline\hline
11a&11&$(2,2)$&11&1&0&2 & $U$ \\
\hline
11b&11&$(2,2)$&11&-&-&2 & $U(11)$\\
\hline
22&22&$(2,0,0,2)$&	22&-&-&6 & \\
&&	&	11&1&0&2 & $U$\\
&&		&	2&10&1&0 & \\
\hline
15a&15&$(2,1,0,2)$&	15&-&-&5 &\\
&&			&			5&2&0&1 & $U(3)\oplus A_2\oplus A_2$ \\
&&			&			3&2&0&2\\
\hline
15b&15&$(2,0,1,4)$&	15&-&-&7 & \\
&&			&			5&1&0&4 & $H_5\oplus A_4$\\
&&			&			3&4&1&1 &\\
\hline
15c&15&$(2,0,2,2)$&	15&-&-&4 & \\
&&			&			5&1&0&4 & $H_5\oplus A_4$\\
&&			&			3&4&0&0 &\\
\hline
30a&30&$(2,0,1,0,0,0,1,1)$&	30&-&-&1\\
&&&				15&-&-&5\\
&&	&			5&2&0&1 & $U(3)\oplus A_2\oplus A_2$\\
&&	&			3&2&0&2\\
&&	&			2&10&0&0\\
\hline
30b&30&$(2,0,0,1,0,0,1,3)$&	30&-&-&3\\
&&	&			15&-&-&7\\
&&	&			5&1&0&4 & $H_5\oplus A_4$\\
&&	&			3&4&1&1\\
&&	&			2&9&1&0\\
\hline
16a&16&$(2,0,0,0,6)$&	16&0&0&6\\
&&	&		8&0&0&6 & $U\oplus D_4$\\
&&	&		4&0&0&6\\
&&	&		2&7&2&0\\
\hline
16b&16&$(2, 0, 0, 2, 4)$ &	16&-&-&4\\
&&	&		8&0&0&6 & $U(2)\oplus D_4$\\
&&	&		4&0&0&6\\
&&	&		2&6&1&0\\
\hline
20&20&$(2,0,1,0,0,2)$&		20&-&-&3\\
&&	&		10&-&-&7\\
&&	&		5&2&0&1 & $U(2)\oplus D_4$\\
&&	&		4&0&0&6\\
&&	&		2&6&1&0\\
\hline
24&24&$(2,0,0,0,0,1,0,4)$&		24&-&-&5\\
&&	&		12&-&-&5\\
&&	&		6&0&0&11 & $U\oplus D_4$\\
&&	&		3&4&1&1\\
&&	&		2&7&2&0\\
\end{tabular}
\vspace{0.2cm}

\caption{Non-symplectic automorphisms with $\varphi(n)=8,10$}
\label{resumen}
\end{table}
}

This classification allows us to prove the following result, which provides 
explicit birational models for a very general pair $(X,\sigma)$
under the previous conditions (see Remark \ref{nsvg} about the generality assumption in the statement).

\begin{theorem}
\label{main}
Let $X$ be a very general complex K3 surface with a purely non-symplectic automorphism $\sigma$  
of order $n\geq 2$ such that $\varphi(n)=8$ or $10$ and $\dim(V^{\sigma})=2$, then up to a birational 
isomorphism $(X,\sigma)$ belongs to the  
families described in Table \ref{1dim}, where $a\in \C$ is a parameter, $\zeta_n$ denotes a primitive $n$th root of unity and  
$(*)$ means: minimal resolution of a degree $11$ covering  of a principal
 homogeneous space of order  $11$ of the rational elliptic surface $y^2=x^3+x+t$ (see Example \ref{ex11b}).
 
{\small
\begin{table}[h!]
\begin{center}
\begin{tabular}{c|l|ll}
$n$ &  $X$ & $\sigma$  \\[2pt]

\hline\hline
&&&\\[-1pt]

\multirow{2}{*}{11}  & a)\quad $y^2=x^3+ax+(t^{11}-1)$&$(x,y,\zeta_{11}t)$ \\[2pt]
& b)\quad (*)&\\[5pt]

\hline
&&&\\[-1pt]

\multirow{2}{*}{15}& a)\quad $y^2=x^3+(t^5-1)(t^5-a)$ & $(\zeta_3 x,y,\zeta_5t)$ \\[2pt]
& b)\quad $y^2=x_0^6+ x_0x_1^5+x_2^6+ax_0^3x_2^3$ & $(x_0,\zeta_5x_1,\zeta_{3}x_2,y)$\\[2pt] 
& c)\quad $y^3=x_0^5x_1+x_1^2x_2^2+x_1^4x_2+ax_1^6$ & $(\zeta_5 x_0,x_1,x_2,\zeta_3 y)$\\[5pt]

\hline
&&&\\[-1pt]

\multirow{2}{*}{16}& a)\quad $y^2=x^3+t^2x+at^3(t^8+1)$ & $(\zeta_{16}^2 x,\zeta_{16}^3y,\zeta_{16}^2t)$   \\[2pt]
 &b)\quad  $y^2=x_0(x_0^4x_2+x_1^5+x_1x_2^4+ax_1^3x_2^2)$ & $(x_0,\zeta_{8}^7 x_1,\zeta_{8}^3x_2,\zeta_{16}^3y)$  \\[5pt]
\hline
&&&\\[-1pt]

20 & $y^2=x_0(x_1^5+x_2^5+x_0^2x_2^3+a x_0^4x_2)$ & $(-x_0,\zeta_5x_1,x_2,iy)$  \\[5pt]

\hline
&&&\\[-1pt]

22 & $y^2=x^3+ax+(t^{11}-1)$ & $(x,-y,\zeta_{11}t)$  \\[5pt]

\hline
&&&\\[-1pt]
24& $y^2=x^3+t(t^4-1)(t^4-a)$ &$(\zeta_{12}x,\zeta_8y,it)$  \\[5pt]

\hline
&&&\\[-1pt]

\multirow{2}{*}{30} & a) \quad $y^2=x^3+(t^5-1)(t^5-a)$ &$(\zeta_3 x,-y,\zeta_5t)$ \\[2pt]
 &  b)\quad $y^2=x_0^6+ x_0x_1^5+x_2^6+ax_0^3x_2^3$ & $(x_0,\zeta_5x_1,\zeta_{3}x_2,-y)$ \\[2pt]
\end{tabular}
   \vspace{0.2cm}
   \caption{One dimensional families of $K3$ surfaces with non-symplectic automorphisms}\label{1dim} 
\end{center}
\end{table}
}
 \end{theorem}

\begin{corollary}
The moduli space of K3 surfaces carrying a purely non-symplectic automorphism of order $n$ has 
a unique one-dimensional component for $n=20,22, 24$,
three irreducible components for $n=15$ and  two irreducible components for $n=11,16, 30$.
\end{corollary}

For orders  $22, 15, 30$ and $20$ we actually prove a stronger version of Theorem \ref{main}, 
since we provide projective models without assuming $X$ to be very general.

Finally, in case $n=22$ and $n=15$ we classify purely non-symplectic automorphisms 
of order $n$, that is we provide the same type of  information contained in Table \ref{resumen},  
without assuming $\dim(V^\sigma)=2$, see Theorem \ref{thm22} and Theorem \ref{thm15}.

The structure of the paper is the following.   
 In the first section we give preliminaries on non-symplectic automorphisms of K3 surfaces and we fix the corresponding 
 notation: fixed loci, invariant lattices and eigenspaces in cohomology, moduli spaces.
In the second section, for each order  $n\in \{11,22,15,30, 16,20,24\}$, 
we prove Theorem \ref{table2} (see Theorem \ref{order11} and Propositions \ref{prop22}, \ref{prop15}, \ref{prop30}, \ref{teo-16}, \ref{prop20}, \ref{prop24}) 
and Theorem \ref{main}.
In the third and fourth section we prove Theorem \ref{thm22} and Theorem \ref{thm15} respectively.

\section{Background and preliminary results} \label{background}
We will work over the complex numbers and  we will denote by $\zeta_i$ a primitive $i$th root of unity.
Let $X$ be a K3 surface over $\C$ and let $\sigma$ be a purely 
non-symplectic automorphism of $X$ of order $n\geq 3$, i.e. 
$\sigma^*(\omega_X)=\zeta_n\omega_X$, where $\omega_X$ is a generator of the 
complex vector space $H^{2,0}(X)$.
\begin{notation}In what follows we will denote by $\sigma_{k}$ an element of  $\langle\sigma\rangle$ whose order is $k$.
\end{notation}

\subsection{Fixed locus}
We start describing the fixed locus of $\sigma$. 
The local action of $\sigma$ in a neighborhood of one of its fixed points can be linearized 
and can be described  by a matrix of the form
\[
A_{i,n}=\left(
\begin{array}{cc}
\zeta_{n}^{i+1}&0\\
0&\zeta_{n}^{n-i}
\end{array}
\right),\quad i=0,1,2,\ldots,\left\lfloor \frac {n-1}2\right\rfloor,
\]
 see \cite[\S 5]{Nikulin}.
 When $i=0$ the fixed point belongs to a fixed curve, otherwise it is an isolated fixed point.
This description implies that the fixed locus of $\sigma$ 
is the union of isolated points and disjoint smooth curves. Moreover, by the Hodge index theorem, 
the fixed locus contains at most one curve of genus $g\geq 2$. In what follows we will use the following notation for the fixed locus of $\sigma$:
\[
\Fix(\sigma)=C_g\sqcup R_1\sqcup \ldots\sqcup R_k\sqcup \{p_1,\dots, p_N\},
\]
where $C_g$ is a smooth curve of genus $g$, $R_1,\dots,R_k$ are smooth rational curves 
and $p_1,\dots,p_N$ are isolated fixed points.
The fixed points such that the local action is given by the matrix 
$A_{i,n}$ will be called points of type $A_{i,n}$ and the number 
of such points will be denoted by $a_{i,n}$. 

We now recall the {\em holomorphic Lefschetz formula} \cite{AS}, which 
relates these numbers with the action of $\sigma^*$ on the cohomology groups 
$H^j(X,\Osh_X)$:
\[
\sum_{j=0}^2 \tr\left(\sigma^*_{|H^j(X,\Osh_X)}\right)=\sum_{i=0}^{\left\lfloor \frac {n-1}2\right\rfloor}
\frac{a_{i,n}}{(1-\zeta_{n}^{i+1})(1-\zeta_{n}^{n-i})}+\alpha \frac{1+\zeta_n}{(1-\zeta_n)^2},
\]
where $\alpha:=\sum_{C\subset \Fix(\sigma)} (1-g(C))$. 
Observe that 
\[
H^1(X,\Osh_X)=H^3(X,\Osh_X)=\{0\}
\]
since $X$ is a K3 surface, $\sigma^*={\rm id}$ on $H^0(X,\Osh_X)\cong \C$ 
and $\sigma^*$ acts as multiplication by $\bar \zeta_n$ on 
$H^2(X,\Osh_X)\cong H^{0,2}(X)=\C\bar\omega_X$.  
Thus the left hand side of the formula is equal to $1+\bar\zeta_n$.

Finally, we recall {\em Hurwitz formula} for a ramified covering $f:X\to Y$ of degree $d$ between smooth 
complex projective varieties, 
which will be used several times in the paper 
for both curves and surfaces:
\[
K_X\sim f^*K_Y+\sum_{i}(e_i-1)C_i,
\]
where the $C_i$'s are the irreducible components of the ramification locus and $e_i$ 
is the associated ramification index (see for example \cite[\S 16, \S 17]{BP}). 

 \subsection{Eigenspaces and invariant lattices}
We now consider  the action of $\sigma^*$ in $H^2(X,\Z)$ and $H^2(X,\C)$. 
We will denote by $S(\sigma^i)\subset H^2(X,\Z)$ the invariant lattice of $\sigma^i$
for $i=0,\dots,n-1$.
Moreover, for any divisor $k$ of $n$ let
\[
H^2(X,\C)^{\sigma}_{k}:=\{x\in H^2(X,\C): \sigma^*x=\zeta_kx\}
\]
and let $d_k$ be its dimension. In particular $d_1$ is the rank of $S(\sigma)$ 
and $d_n$ is the dimension of $V^{\sigma}=H^2(X,\C)^{\sigma}_n$. 
In what follows we will denote by $d$ the vector whose entries are the numbers 
$d_k$, as $k$ varies in the set of divisors of $n$ in decreasing 
order: 
\[
d=(d_n,\dots,d_k,\dots, d_1),\ k|n.
\]

\begin{remark}\label{ns}
Observe that, since $\sigma$ is purely non symplectic, then $S(\sigma^i)$ 
is contained in the N\'eron-Severi lattice of $X$ for any $i=0,\dots, n-1$. 
In fact, given $x\in S(\sigma^i)$ we have
\[
(x,\omega_X)=((\sigma^i)^*x,(\sigma^i)^*\omega_X)=(x,\zeta_n^i\omega_X)=\zeta_n^i(x,\omega_X),
\]
which implies $(x,\omega_X)=0$ and thus $x\in H^2(X,\Z)\cap \omega_X^{\perp}={\rm NS}(X)$.
\end{remark}

We also recall the {\em topological Lefschetz formula} \cite[Theorem 4.6]{AS}, for simplicity we state it only for $\sigma$: 
\[
\chi(\Fix(\sigma))=\displaystyle\sum_{i=0}^4(-1)^{i}\tr\left(\sigma^*\Big|_{H^{i}(X,\R)}\right),
\]
where the right side is equal to $2+\tr(\sigma^*\Big|_{H^{2}(X,\R)})$ since $H^i(X,\R)=\{0\}$ for $i=1,3$ and 
$\sigma^*={\rm id}$ on $H^i(X,\R)$  for $i=0,4$.

Finally, we recall some notation for lattices which will appear in the paper: $A_\ell$ ($\ell\geq 1$), $D_m$ ($m\geq 4$) 
and $E_n$ ($n=6,7,8$) denote the negative definite even lattices associated to the Dynkin diagrams 
of the corresponding types, $U$ and $H_5$ denote the lattices with the following Gram matrices
\[
U=\left(
\begin{array}{cc}
0 & 1\\
1 & 0
\end{array}
\right),\quad 
H_5=\left(
\begin{array}{cc}
2 & 1\\
1 & -2
\end{array}
\right),
\]
and $U(r)$ with $r\geq 2$ denotes the lattice whose Gram matrix is the one of $U$ multiplied by $r$.

\subsection{Moduli spaces}\label{moduli}
Let $X$ be a K3 surface with an order $n$ automorphism $\sigma$  
such that $\sigma^*(\omega_X)=\zeta_n\omega_X$.
The period line $\C\omega_X$  belongs to the domain
\[
\mathcal D^{\sigma}=\{\C z\in \P(V^{\sigma}):  (z,\bar z) >0, (z,z)=0\},
\]
where $V^{\sigma}$ is the $\zeta_n$-eigenspace of $\sigma^*$ in $H^2(X,\C)$.
Observe that, for $n\geq 3$, we have 
\[
(z,z)=(\sigma^*z,\sigma^*z)=\zeta_n^2(z,z),
\]
thus the condition $(z,z)=0$ is not necessary and 
$\mathcal D^\sigma$ can be easily proved to be isomorphic to a complex ball.
On the other hand, if $n=2$, then $\mathcal D^{\sigma}$ is a type $IV$ 
Hermitian symmetric space.
By  \cite[Theorem 11.3]{DK} an arithmetic quotient of a Zariski open subset of 
$\mathcal D^{\sigma}$ parametrizes isomorphism classes of 
$(\rho,M)$-polarized K3 surfaces, where $\rho:C_n\to O(L_{K3})$ is a representation 
induced by the isometry $\sigma^*$  of $H^2(X,\C)$ 
and the choice of an isometry $H^2(X,\Z)\to L_{K3}$, 
and $M\subseteq L_{K3}$ is the invariant lattice of ${\rm Im}(\rho)$. 
In particular such moduli space has dimension $\dim(\mathcal D^{\sigma})=\dim(V^{\sigma})-1$ if $n\geq 3$ 
and  $\dim(V^{\sigma})-2$ if $n=2$.

On the other hand, if $T_X$ is the transcendental lattice of $X$,  it is known that the eigenvalues of $\sigma^*$ 
in $T_X\otimes_{\mathbb Z} \C$  are the primitive $n$th roots of unity  \cite[Section 3]{Nikulin}, thus 
$\rk(T_X)=\dim(V^{\sigma})\varphi(n)$. 
Since $\rk(T_X)\leq 21$, this implies that 
\[
\dim(\mathcal D^\sigma)\leq \gamma(n):=\left\lfloor \frac{21}{\varphi(n)}\right\rfloor-1.
\] 
In particular the dimension of $\mathcal D^\sigma$ is at most one if $\gamma(n)=1$. 
 We show that the converse also holds.

 \begin{lemma}\label{moduli}
 Let $n\not=60$ be a positive integer with $\varphi(n)\leq 20$ and $\gamma(n)>1$, 
 then there exist a K3 surface $X$ and a purely non-symplectic automorphism $\sigma$ 
 of $X$ of order $n$ such that $\dim(\mathcal D^\sigma)>1$.
\end{lemma}

\begin{proof} We will denote by $d(n)$ the dimension of the moduli space of K3 surface carrying a purely non-symplectic automorphism of order $n$.
The orders $n\geq 2$ with $\varphi(n)\leq 20$ and $\gamma(n)>1$ are $n=7,9, 14, 18$ with $\gamma(n)=2$, 
 $n=5,8,10,12$ with $\gamma(n)=4$, $n=3,4,6$ with $\gamma(n)=9$, and $n=2$.
  
 For prime orders $n=3,5,7$ it is known by \cite{ArtebaniSartiTaki} that $d(n)=\gamma(n)$.
 Moreover, the same is true for orders $n=6, 10, 14$ by Proposition \ref{gs} and \cite{ArtebaniSartiTaki}.
 
 For order $n=9$ it is known by \cite{ACV} that $d(n)=2$.
 Moreover, the general member of one of its components of maximal dimension is an elliptic K3 surface 
 with Weierstrass equation
 \[
 y^2=x^3+t(t^3-a)(t^3-b)(t^3-c),\ a,b,c\in \C,
 \]
 which carries the order nine automorphism  $\sigma(x,y,t)=(\zeta_9^4x, \zeta_9^6y, \zeta_3t)$. 
 This surface also admits the non-symplectic involution 
 $\tau(x,y,t)=(x,-y,t)$ which commutes with $\sigma$, so it carries the non-symplectic automorphism 
 $\sigma\tau$  of order $18$. This shows that $d(18)=2$ as well.  
 
 When $n=4$, \cite[Example 6.3]{ArtebaniSarti4} is a $9$-dimensional family of K3 surfaces with a purely non-symplectic automorphism 
 of order $4$. 
 
 When $n=8$, \cite[Example 4.1]{TabbaaSarti} is a $2$-dimensional family of K3 surfaces with a purely non-symplectic automorphism 
 of order $8$.
 
 When $n=12$ the family of elliptic K3 surfaces defined by the Weierstrass equation 
 \[
 y^2=x^3+t\prod_{i=1}^5(t^2-a_i),\ a_i\in \C
 \]
 is $4$-dimensional and has an order $12$ automorphism, 
 $\sigma(x,y,t)=(-\zeta_3x,iy,-t)$ which can be easily checked to be purely non-symplectic.

When $n=2$ it is well known that  $d(2)=19$ and there is a unique component of maximal dimension 
 whose  general element is a double cover of $\P^2$ branched along a smooth plane sextic.
 \end{proof}
 
 \begin{remark}\label{nsvg}
 Under the hypotheses of Theorem \ref{main}, since $T_X$ has the structure of a $\Z[\zeta_n]$-module by \cite[Section 3]{Nikulin} and $\dim(V^\sigma)=2$
we have that ${\rm rk }\,{\rm NS(X)}\geq 22-2\varphi(n)$. 
The generality assumption in the statement of the theorem means that the N\'eron-Severi lattice of $X$ 
has the minimal rank. 
\end{remark}
 
 Finally, we recall a result contained in \cite[Theorem 1.4, Theorem 1.5]{GS} and in \cite{Dillies}.

\begin{proposition}\label{gs}
Let $X$ be a K3 surface with a non-symplectic automorphism  $\sigma$ of order $n$.
If either 
\begin{enumerate}[i.]
\item $n=5,13,17,19$, 
\item or $n=7,11$ and the fixed locus of $\sigma$ contains a curve, 
\item or $n=3$ and the the fixed locus of $\sigma$ contains at least two curves, 
\item or $n=3$ and the the fixed locus of $\sigma$ contains a curve and two points,
\end{enumerate}
then $X$ admits a non-symplectic automorphism $\tau$ of order $2n$ with $\tau^2=\sigma$.

Moreover, if $n = 11$ and  the fixed locus of $\sigma$ consists of only isolated fixed points, 
then $X$ does not admit a non-symplectic automorphism $\tau$ of order $22$ with $\tau^2=\sigma$.
\end{proposition}

 \section{Proof of Theorem \ref{table2} and Theorem \ref{main}}\label{s2}
 In this section we prove the two main theorems for each order.

\subsection{Order 11}\label{sec:11}
Non-symplectic automorphisms of order $11$ have been classified in \cite{OZ2} and \cite[Sec. 7]{ArtebaniSartiTaki}.
In particular the proof of Theorem \ref{main} for order $11$ follows from the following result.

\begin{theorem}\label{order11}
Let $X$ be a K3 surface with a non-symplectic automorphism $\sigma$ of order $11$ such that $\rk \,S(\sigma)=2$ 
(or equivalently $\dim(V^\sigma)=2$).
Then two cases can occur:
\begin{enumerate}[a)]
\item   $\Fix(\sigma)=C_1\sqcup\{p_1,p_2\}$ and $S(\sigma)={\rm NS}(X)\cong U$,
\item $\Fix(\sigma)=\{p_1,p_2\}$ and $S(\sigma)={\rm NS}(X)\cong U(11)$.
\end{enumerate}
where $C_1$ is a smooth curve of genus one.
In both cases $d=(2,2)$. Moreover, up to birational isomorphisms, $(X,\sigma)$ belongs to the family in Example \ref{ex11a} in case a)  and 
 to the family in Example \ref{ex11b} in case b).
\end{theorem}

\begin{example}\label{ex11a}
Given $a\in \C$, let $X_{11a}$ be the elliptic fibration with Weierstrass equation
\[
y^2=x^3+ax+(t^{11}-1).
\]
For general $a\in \C$ the fibration has one fiber of Kodaira type $II$ over $t = \infty$ 
and $22$ fibers of type $I_1$. Observe that $X_{11a}$ carries the order $11$ automorphism
\[
\sigma_{11a}(x, y,t) = (x, y,\zeta_{11}t),
\]
which fixes the smooth fiber over $t=0$ and two points in the fiber over $t=\infty$.
\end{example}

\begin{example}\label{ex11b}
Consider the extremal rational elliptic surface $\phi:Y\to \P^1$ with Weierstrass equation
\[
y^2=x^3+x+t.
\]
The fibration has a fiber of type $II^*$ over $t=\infty$ 
and two fibers of type $I_1$ over the zeroes of $\Delta=4+27t^2$
thus it is extremal.
Given $\alpha\in \P^1$ such that $\phi^{-1}(\alpha)$  is smooth, 
let $\phi_{\alpha,e}:Y_{\alpha,e}\to\P^1$ be the principal homogeneous space of $\phi$ 
associated to a non-trivial $11$-torsion element $e$ in $\phi^{-1}(\alpha)$.
We recall that  $\phi_{\alpha,e}$ has the same configuration of singular fibers 
as $\phi$ and it has a fiber $F=11F_0$ of multiplicity $11$ over $\alpha$ such 
that $(F_0)_{|F_0}=e\in \Pic^0(F_0)$ (see \cite[\S 4, Chapter V]{CD} as a reference for 
principal homogeneous spaces of rational jacobian elliptic fibrations). 
Let $\psi_{\alpha,e}:Z_{\alpha,e}\to\P^1$ be the degree $11$ base change 
of $\phi_{\alpha,e}$ branched along $t=\infty$ and $t=\alpha$.
A minimal resolution of $Z_{\alpha,e}$ is a K3 surface $X_{\alpha,e}$ carrying an elliptic fibration 
$\pi_{\alpha,e}$ induced by $\psi_{\alpha,e}$ which has $22$ fibers of type $I_1$ over the 
two fibers of type $I_1$ of $\psi_{\alpha,e}$ 
and a fiber of type $II$ over $t=\infty$.
The covering automorphism of $Z_{\alpha,e}\to Y_{\alpha,e}$ induces an order $11$ automorphism $\sigma_{11b}$
of $X_{\alpha,e}$.
\[
\xymatrix{
X_{\alpha,e}\ar[r]\ar[d]^{\pi_{\alpha,e}} &Z_{\alpha,e}\ar[d]^{\psi_{\alpha,e}}\ar[r] & Y_{\alpha,e}\ar[d]^{\phi_{\alpha,e}}\\
\P^1\ar[r] & \P^1\ar[r]^{11:1} & \P^1.
}
\]
We will denote by $(X_{11b},\sigma_{11b})$ the family of K3 surfaces 
with automorphism obtained with this construction.
The automorphism $\sigma_{11b}$ fixes exactly two points in the fiber of $\pi_{\alpha,e}$ of type $II$.
\end{example}

\subsection{Order 22}\label{sec:22}
In this section we will give the classification of purely non-symplectic automorphisms of order $22$ with $\dim (V^{\sigma})=2$. 
The full classification,  including the cases with $\dim(V^\sigma)=1$, will be given in Section \ref{class22}.
\begin{proposition}\label{prop22}
Let $X$ be a K3 surface with a purely non-symplectic automorphism $\sigma$ of order $22$ 
such that $\dim (V^\sigma)=2$. Then the fixed loci of $\sigma=\sigma_{22}$ and of its powers 
$\sigma_{11}=\sigma^2$ and $\sigma_2=\sigma^{11}$  
are as follows:
\[
\begin{array}{c|c|c}
\Fix(\sigma_{22}) & \Fix(\sigma_{11}) & \Fix(\sigma_{2})\\
\hline
\{p_1,\dots,p_6\} & C_1\sqcup \{p_5,p_6\} & C_{10}\sqcup R \\
\end{array}
\]
where $g(C_1)=1,\ g(C_{10})=10$ and $g(R)=0$.
Moreover $d=(2,0,0,2)$ and ${\rm NS}(X)\cong U$ for a very general K3 surface with such property.
\end{proposition}

\begin{proof} Decomposing $H^2(X,\C)$ as the direct sum of the eigenspaces 
of $\sigma^*$ we obtain, with the notation in Section \ref{background}:
\[
\dim H^2(X,\C)=22=10d_{22}+10d_{11}+d_2+d_1=20+10d_{11}+d_2+d_1.
\]
Since $d_{22}=2$, thus $d_1+d_2=2$ and $d_{11}=0$, so either $d=(2,0,1,1)$ or $(2,0,0,2)$.
Let $\chi_i:=\chi(\Fix(\sigma_{i})), i\in\{2,11,22\}$. By the topological Lefschetz formulas we have

\begin{equation}
\begin{cases}
\chi_{22}&=d_{22}-d_{11}-d_{2}+d_1+2\\
\chi_{11}&=-d_{22}-d_{11}+d_{2}+d_1+2\\
\chi_{2}&=-10d_{22}+10d_{11}-d_{2}+d_1+2.
\end{cases}
\label{sistema22}
\end{equation}

This implies $\chi_{11}=2$.
By Proposition \ref{gs}, if a K3 surface admits a non-symplectic automorphism of order $11$
without fixed curves,  it does not admit a non-symplectic automorphism of order $22$.
This result and Theorem \ref{order11} imply that $\Fix(\sigma_{11})$ is the union of a 
smooth genus one curve $C$ and two points $p,q$. 
On the other hand the same equations give that 
$\chi_{22}=4$ if $d=(2,0,1,1)$ and $=6$ if $d=(2,0,0,2)$.
This implies that  $\sigma_{22}$ is not the identity on $C$, thus 
it acts on it as an involution with $4$ fixed points, and it either exchanges or fixes $p$ and $q$.

We will now show that $\sigma_{22}$ must fix $p$ and $q$, i.e. that $\chi_{22}=6$.
Observe that the fixed points of $\sigma_{22}$ on $C$ are of type $A_{10,22}$ 
since they are contained in a fixed curve of $\sigma_{22}^2$. 
If these were the only fixed points of $\sigma_{22}$, an easy computation shows that  
holomorphic Lefschetz formula does not hold, giving a contradiction.

Finally $\chi_{2}=-16$. By \cite{Nikulin} this implies that the fixed locus of $\sigma_2$ 
is either a genus $9$ curve or the union of a genus $10$ curve and a rational curve.
The first case is not possible since a curve of genus $9$ has no order $11$ automorphisms 
by the Riemann-Hurwitz formula.

Observe that for a very general $K3$ surface as in the statement 
${\rm rk}\, {\rm NS}(X)=22-2\varphi(22)=2$ (see Remark \ref{nsvg})
and  $S(\sigma_{11})\subseteq {\rm NS}(X)$ by Remark \ref{ns},
thus ${\rm NS}(X)=S(\sigma_{11})\cong U$ by Theorem \ref{order11}.
\end{proof}

\begin{example}\label{ex22}
 The elliptic K3 surface in Example \ref{ex11a}  
 \[
y^2=x^3+ax+(t^{11}-1), \ a\in\C
\]
admits the order $22$ automorphism  
\[
\sigma_{22}(x,y,t)=(x,-y,\zeta_{11}t),
\]
which fixes four points in the smooth fiber over $t=0$ and 
two points in the fiber of type $II$ over $t=\infty$.
The involution $\sigma_2=\sigma_{22}^{11}$ fixes the curve $y=0$, which has genus $10$,
and the sections at infinity. 
 Since $\sigma_2$ has fixed curves and since there exist no symplectic 
 automorphism of a K3 surface of order $11$ \cite{Nikulin}, then $\sigma$ is purely non-symplectic.
\end{example}

\begin{proof}[Proof of Theorem \ref{main}, order 22]
Let $X$ be a K3 surface with a purely non-symplectic automorphism $\sigma=\sigma_{22}$ of order $22$.
By Proposition \ref{prop22}  $\Fix(\sigma_{11})$ contains an elliptic curve $C_1$ and two points. 
Thus by Theorem \ref{order11} $(X,\sigma_{11})$ belongs to the family in Example \ref{ex11a} up to isomorphism, i.e. 
it carries an elliptic fibration $\pi:X\to \P^1$ with Weierstrass equation
\[
y^2=x^3+ax+(t^{11}-1), \ a\in\C
\]
and $\sigma_{11}(x,y,t)=(x,y,\zeta_{11}t)$.
The lattice generated by the class of a fiber and 
the class of a section of $\pi$ is isometric to the lattice $U$ 
and is fixed by the automorphism $\sigma_{11}^*$, thus
it coincides with $S(\sigma_{11})$ by Theorem \ref{order11}.
Since $\sigma_{22}^*$ preserves the lattice $S(\sigma_{11})$ 
and this contains a unique class of elliptic fibration 
and a unique class of smooth rational curve, 
then $\sigma_{22}^*$ preserves both.
By Proposition \ref{prop22} the fixed locus of the 
involution $\sigma_2$ is the disjoint union of a smooth curve $C_{10}$ 
of genus $10$ and a smooth rational curve $R$. 
The curve $C_{10}$ is clearly transverse to the fibers of $\pi$,
thus each fiber of $\pi$ contains fixed points of $\sigma_2$.
This implies that the action induced by $\sigma_2$ on $\P^1$ is the identity,
i.e. each fiber of $\pi$ is preserved by $\sigma_2$.
Moreover, the unique section $S$ of $\pi$ must be pointwise fixed 
by $\sigma_2$, so that $R=S$.
Since $\sigma_2$ is an involution which preserves each fiber of $\pi$ and fixes  
$S$, then it is defined by $(x,y,t)\mapsto (x,-y,t)$.
This shows that the action of $\sigma_{22}=\sigma_{11}\circ \sigma_2$ 
on $\pi$ is the one described in the statement of Theorem \ref{main}, 
concluding the proof.
\end{proof}

\subsection{Order 15}\label{sec:15}
In this section we will give the classification of purely non-symplectic automorphisms of order $15$ 
with $\dim(V^\sigma)=2$. The full classification, including the cases with $\dim(V^\sigma)=1$, 
will be given in Section \ref{class15}.

\begin{proposition}\label{prop15}
Let $X$ be a K3 surface with a purely non-symplectic automorphism $\sigma$ of order $15$ 
such that $\dim(V^\sigma)=2$. Then the fixed loci of $\sigma=\sigma_{15}$ and its powers 
$\sigma_{i}=\sigma^{\frac{15}{i}}$ 
are as follows:
 \[
\begin{array}{c||c|c|c}
&\Fix(\sigma_{15}) & \Fix(\sigma_{5}) & \Fix(\sigma_{3})\\
\hline
a)&\{p_1,\dots,p_5 \}& C_2\sqcup \{p_1\} & C_2'\sqcup\{p_2,p_3\} \\
b)&\{p_1,\dots,p_7 \}& C_1\sqcup \{p_1,\dots,p_4\} & C_4\sqcup R\sqcup \{p_1\}\\
c) & \{p_1,\dots,p_4\} & C_1\sqcup \{p_1,q_1,q_2,q_3\} & C_4
\end{array}
\]
where $g(C_1)=1, g(C_2)=g(C_2')=2, g(C_4)=4$ and $g(R)=0$.
Moreover, $d=(2,1,0,2)$ in case a), $d=(2,0,1,4)$ in case b) 
and $d=(2,0,2,2)$ in case c).
Finally ${\rm NS}(X)\cong U(3)\oplus A_2\oplus A_2$ for a very general K3 surface $X$ in case a) 
and ${\rm NS}(X)\cong H_5\oplus A_4$ for a very general K3 surface $X$ in cases b) and c), 
where $H_5$ is the lattice defined in \cite[Section 1]{ArtebaniSartiTaki}. 
\end{proposition}

\begin{proof} Decomposing $H^2(X,\C)$ as the direct sum of the eigenspaces 
of $\sigma^*$ we obtain, with the notation in Section \ref{background}:
\[
22=8d_{15}+4d_{5}+2d_{3}+d_1=16+4d_{5}+2d_{3}+d_1,
\]
thus $d\in \{(2,1,0,2), (2,0,2,2), (2,0,1,4), (2,0,0,6)\}$.
Let $\chi_{i}:=\chi(\Fix(\sigma_{i})), i\in\{3,5,15\}$. By the topological Lefschetz fixed point formulas:
 \begin{equation}
\begin{cases}
\chi_{15}&=d_{15}-d_{5}-d_{3}+d_1+2\\
\chi_5&=-2d_{15}-d_{5}+2d_3+d_1+2\\
\chi_3&=-4d_{15}+4d_5-d_3+d_1+2.
\end{cases}
\label{sistema15}\end{equation}
We will show that $d=(2,1,0,2)$, $d=(2,0,1,4)$ and $d=(2,0,2,2)$ are the only possible cases.

Assume that $d=(2,1,0,2)$. Thus $(\chi_{15}, \chi_5, \chi_3)=(5,-1,0)$.
By \cite{ArtebaniSartiTaki} we have that  $\Fix(\sigma_{5})$ is the union of a curve $C_2$
of genus $2$ and one point. Since $\chi_{15}=5$ the action of $\sigma$ on $C_2$ 
has order $3$  with $4$ fixed points by the Riemann-Hurwitz formula. In particular 
$\Fix(\sigma)$ is the union of $5$ points.
Finally by \cite{ArtebaniSarti}  $\Fix(\sigma_{3})$ is either the union of a genus $2$ curve and $2$ points 
or contains a curve of genus three. 
The second case is not possible since there is no genus $3$ curve with an order five automorphism 
by \cite[Table 5]{broughton}.

If $d\not=(2,1,0,2)$, then $\chi_{5}=4$ 
and $\chi_{3}=-6,-3,0$ if $d=(2,0,2,2), (2,0,1,4)$ or $(2,0,0,6)$ respectively.
By \cite{ArtebaniSartiTaki}  $\Fix(\sigma_{5})$ is either the union of an elliptic curve $C_1$ and $4$ points,
or the union of $4$ points. 
Observe that $C_1$ can not be contained in $\Fix(\sigma_{3})$ since 
by \cite{ArtebaniSarti} this would imply $\chi_{3}\geq 3$. 
Thus, looking at the possible actions of $\sigma$ on $C_1$ and the $4$ points we find that  
$\chi_{15}$ is either $1$, $4$ or $7$.

If $d=(2,0,0,6)$, then $\chi_{15}=10$ by \eqref{sistema15}, giving a contradiction.

If $d=(2,0,1,4)$, then $\chi_{15}=7$ by \eqref{sistema15}.
Thus $\Fix(\sigma_5)$ is the union of an elliptic curve $C_1$ and $4$ points and 
$\Fix(\sigma)$ consists of $7$ points, $3$ of them on $C_1$.
Moreover $\chi_{3}=-3$, thus by \cite{ArtebaniSarti} $\Fix(\sigma_{3})$ is 
either the union of a genus $4$ curve, a rational curve and one point, or 
it contains a curve of genus $3$. The last case is not possible by \cite[Table 5]{broughton}.

If $d=(2,0,2,2)$, then $\chi_{15}=4$ and $\chi_3=-6$ by \eqref{sistema15}.
Observe that $\sigma$ has $7$ types of isolated fixed points.
The fixed points of type $A_{1,15}, A_{4,15}, A_{7,15}$ are isolated fixed points for $\sigma_3$ too, 
while points of type $A_{2,15},A_{3,15},A_{5,15},A_{6,15}$ lie on a curve fixed by $\sigma_3$.
Observe that $\sigma_{15}$ acts on the set of isolated fixed points of $\sigma_3$ 
with orbits of length either $1$ or $5$.  Thus we have
\[
a_{1,15}+a_{4,15}+a_{7,15}\leq a_{1,3}, \quad a_{1,15}+a_{4,15}+a_{7,15}\equiv a_{1,3} \mod 5.
\]

Moreover, points of type $A_{4,15}, A_{5,15}$ lie on a curve fixed by $\sigma_5$, 
while points of type $A_{1,15}, A_{2,15}, A_{3,15}, A_{6,15}, A_{7,15}$ are isolated fixed points for $\sigma_5$ too.
Checking types one has 
\begin{equation}\label{eqpts}
a_{1,15}+a_{3,15}+a_{6,15}\leq a_{1,5},\quad 
a_{2,15}+a_{7,15}\leq a_{2,5}.  
\end{equation}

Since $\chi_3=-6$,  then $a_{1,3}=0$ by \cite{ArtebaniSarti}.
Applying the holomorphic Lefschetz formula to $\sigma$ with this condition and 
using the fact that $\alpha=0$,  we find that $(a_{1,15},a_{2,15},\ldots,a_{7,15})=(0,1,0,0,3,0,0)$.
Since $a_{5,15}=3$, then we find that $\Fix(\sigma_5)$ contains an elliptic curve $C_1$
and $\sigma$ fixes three points on it.

For $\Fix(\sigma_3)$ there are two possibilities by \cite{ArtebaniSarti}:
it is either a curve of genus four or the union of a genus five curve and a rational curve.
The second case is excluded by Lemma \ref{g50}, where $C'$ is the elliptic curve $C_1$.

We now compute the N\'eron-Severi lattice of a  very general $X$ in each case. Observe that
since $d_{15}=2$ and $\varphi(15)=8$, the N\'eron-Severi lattice of $X$ has rank $22-2\cdot 8=6$.
In case a) the invariant lattice $S(\sigma^5)=S(\sigma_3)$ has rank $d_1+4d_5=6$,
thus ${\rm NS}(X)=S(\sigma_3)\cong U(3)\oplus A_2\oplus A_2$, where the last isomorphism is by \cite{ArtebaniSartiTaki}.
In cases b) and c) the invariant lattice $S(\sigma^3)=S(\sigma_5)$ has rank $d_1+2d_3=6$,
thus we conclude as before that ${\rm NS}(X)=S(\sigma_5)$. In both cases $S(\sigma_5)$ 
is isomorphic to $H_5\oplus A_4$  by \cite{ArtebaniSartiTaki}.
\end{proof}

\begin{lemma}\label{g50}
Let $X$ be a K3 surface and $\tau$ be a non-symplectic automorphism of order three of $X$ 
whose fixed locus is the disjoint union of a smooth curve $C$ of genus five and a smooth rational curve $R$.
Then $X$ has no purely non-symplectic automorphism $\sigma$ of order $15$  such that $\sigma^5=\tau$ 
and such that the fixed locus of $\sigma^3$ contains a curve $C'$ distinct from $C$ and $R$.
\end{lemma}

\begin{proof}
Let $\pi:X\to Y$ be the quotient morphism by $\tau$.
Since the fixed locus of $\tau$ is a smooth curve and the automorphism is non-symplectic of order $3$, 
then $Y$ is a smooth rational surface. 
Moreover,  the invariant lattice $S(\tau)$  has rank two  by \cite{ArtebaniSarti}.
Since $\pi^*$ is injective and $\pi^*{\rm NS}(Y)$ is a sublattice of $S(\tau)$,
then $\rk\,{\rm NS}(Y)\leq 2$. Thus $Y$ is isomorphic to either $\P^2$ or a Hirzebruch surface $\mathbb F_r$, $r\geq 0$.
The covering $\pi$ is branched along a smooth curve $B$
 whose class $[B]\in {\rm NS}(Y)$ satisfies $-3K_Y=2[B]$ by the Hurwitz formula.
This excludes the case $Y\cong \P^2$.
We recall that if $Y=\mathbb F_r$, then $-K_Y=(r+2)f+2e$, where $f^2=0, e^2=-r$ and $f\cdot e=1$.
Thus $r$ must be even and 
\[
[B]=\frac{3(r+2)}{2}f+3e.
\]
Since $R^2<0$, then its image in $Y$ has the same property and is the unique 
curve of negative self-intersection in $Y$, i.e. $[\pi(R)]=e$.
Moreover, since $B$ is the disjoint union of $\pi(R)$ and 
 $\pi(C)$ then
\[
([B]-e)\cdot e=\frac{3(r+2)}{2}-2r=\frac{6-r}{2}=0.
\]
Thus $Y\cong \mathbb F_6$ and the class of $\pi(C)$ is $12f+2e$. 
Let $p:\F_6\to \P^1$ be the natural fibration. Observe that the restriction 
of $p$ to $\pi(C)$ is a double cover of $\P^1$ since $(12f+2e)\cdot f=2$,
thus $\pi(C)$ is hyperelliptic. This implies that there are $12$ fibers of $p$ 
which are tangent to $\pi(C)$.

Assume now that $X$ has an automorphism $\sigma$ of order $15$ with $\sigma^5=\tau$. 
Then $\sigma$ induces an automorphism $\bar\sigma$
of $Y$ of order $5$ which preserves both $\pi(R)$ and $\pi(C)$.
Since $\sigma^3$ is not the identity on $R$, then $\bar\sigma$ is not the identity on $\pi(R)$, 
thus we can assume that it acts on the basis of the fibration $p$ 
as $(x,y)\mapsto (\zeta_5x,y)$, where $\zeta_5$ is a primitive $5$-th root of unity.
In particular there are exactly two fibers of $p$ which are invariant for $\bar\sigma$.

This implies that the image of the
curve $C'$ in $\Fix(\sigma^3)$ 
is a fiber of  $p$, which is invariant for $\bar\sigma$.
On the other hand 
$\bar\sigma$ preserves $\pi(C)$ and thus 
permutes the $12$ fibers of $p$ 
which are tangent to $\pi(C)$.
Thus it should preserve at least two of them.
The curve $\pi(C)$ can not be tangent to $\pi(C')$,
since otherwise $C'$ would be singular, thus 
$\bar\sigma$ should leave invariant three fibers 
of $p$, a contradiction.
\end{proof}

\begin{example}\label{ex15a} 
Let $B$ be the plane sextic defined by
\[
F_6(x_0,x_1,x_2)=a_1x_0^6+ a_2x_0x_1^5+a_3x_2^6+a_4x_0^3x_2^3,
\]
with general $a_1,a_2,a_3,a_4\in \C$. 
Let $X$ be a double cover of $\P^2$ branched along $B$, 
which can be defined by $x_3^2-F_6(x_0,x_1,x_2)=0$ in $\P(1,1,1,3)$.
Then $X$ is a K3 surface carrying an order $15$ 
automorphism
\[
\sigma_{15}(x_0,x_1,x_2,x_3)=(x_0,\zeta_5 x_1,\zeta_3x_2,x_3)
\]
whose fixed locus is the union of $5$ points, which project to the 
points $(1,0,0)$, $(0,1,0)$, $(0,0,1)$ of $\P^2$. 
Observe that $\sigma_5$ 
fixes the genus two curve defined by $x_1=0$ and the point $(0,1,0,0)$,
while $\sigma_3$ fixes the genus two curve $x_2=0$ and the points $(0,0,1,\pm 1)$.
Since both $\sigma_3$ and $\sigma_5$ fix curves, then none of them is symplectic by \cite{Nikulin},
thus $\sigma_{15}$ is purely non-symplectic. 
This is an example of case a) in Proposition \ref{prop15}.
\end{example}

\begin{example}\label{ex15b}
Consider the elliptic surface with Weierstrass equation
\[
y^2=x^3+(t^5-1)(t^5-a),
\]
with general $a\in \C$. Then $X$ is a K3 surface with the 
automorphism of order $15$:
\[
\sigma_{15}(x,y,t)=(\zeta_3 x, y,\zeta_5 t).
\]
The elliptic fibration has one fiber of type $IV$ over $t=\infty$ 
and $10$ fibers of type $II$. 
The automorphism $\sigma_3$   fixes 
the genus $4$ curve defined by $x=0$, the section at infinity 
and the center of the fiber of type $IV$.
The automorphism $\sigma_5$   fixes 
the smooth fiber over $t=0$ and four points in the fiber of $t=\infty$.
The automorphism $\sigma_{15}$ fixes $3$ points in the fiber over $t=0$ 
and $4$ points in the fiber over $t=\infty$. 
As in the previous Example, $\sigma_3$ and $\sigma_5$ fix curves, thus they are non-symplectic 
and $\sigma_{15}$ is purely non-symplectic.
This is an example of case b) in Proposition \ref{prop15}.
\end{example}

\begin{example}\label{ex15c}

Consider $P=\P(1,1,2)$ with coordinates $x_0,x_1,x_2$ and let $D$ 
be a curve of degree $6$ in $P$ of equation 
\[
G_6(x_0,x_1,x_2)=x_0^5x_1+a_1x_1^2x_2^2+a_2x_1^4x_2+a_3x_1^6+a_4x_2^3=0,
\]
where $a_1,a_2,a_3,a_4\in\C$ are general.
Observe that $D$ is smooth, since it does not pass through 
the singular point $(0,0,1)$ and its partial derivatives only vanish at the origin.
Let $Y\cong \F_2$ be the blow up of the singular point of $P$ and 
let $B$ be the preimage of $D$ in $Y$. 
Since $2[D]\sim -3K_P$ and the resolution $Y\to P$ is crepant,
then $2[B]\sim -3K_{Y}$.
Let $X$ be the triple cover of $Y$ branched along $B$. 
By the Hurwitz formula $X$ is a K3 surface.
Observe that the curve $D$ has the order $5$ automorphism
\[
(x_0,x_1,x_2)\mapsto (\zeta_5x_0,x_1,x_2),
\]
which lifts to an order five automorphism $\varphi$ of $X$.
The composition of $\varphi$ with the covering automorphism 
of $X\to Y$ is an order $15$ automorphism $\sigma$ of $X$.
A birational model of $(X,\sigma)$ in $\P(1,1,2,2)$ is 
\[
x_3^3+G_6(x_0,x_1,x_2)=0,\ \sigma(x_0,x_1,x_2,x_3)=(\zeta_5 x_0,x_1,x_2,\zeta_3 x_3).
\]
Embedding $\P(1,1,2,2)$ in $\P^4$ via the map 
$
(x_0,x_1,x_2,x_3)\mapsto (x_0^2,x_0x_1,x_1^2,x_2,x_3),
$
we also obtain a birational model of $(X,\sigma)$ as complete intersection in $\P^4$ 
with three $A_1$ singularities at $q_i=(0,0,0,\zeta_3^i,1)$, $i=1,2,3$:
\[
\left\{
\begin{array}{l}
y_1^2-y_0y_2=0,\\
y_4^3+y_0^2y_1+a_1y_2y_3^2+a_2y_2^2y_3+a_3y_2^3+a_4y_3^3=0,
\end{array}
\right.
\]
\[
\sigma(y_0,y_1,y_2,y_3,y_4)=(\zeta_5^2y_0,\zeta_5y_1,y_2,y_3,\zeta_3y_4).
\]
The fixed locus of $\sigma_3$ is the genus four curve $y_4=0$.
The fixed locus of  $\sigma_5$ is the union of the curve $C_1$ of genus one defined by $y_0=y_1=0$, 
the point $p_1=(1,0,0,0,0)$ and three points over $q_1,q_2,q_3$.
Finally, $\sigma$ fixes $p_1$ and three points in $C_1\cap C_4$.
This is an example of case c) in Proposition \ref{prop15}.
\end{example}

\begin{proof}[Proof of Theorem \ref{main}, order 15]
Let $X$ be a K3 surface with a purely non-symplectic automorphism $\sigma_{15}$ of order $15$.
By Proposition \ref{prop15},
$\Fix(\sigma_{15})$ contains either $4,5$ or $7$ isolated fixed points.

{\em Case a)}. We first assume that $\Fix(\sigma_{15})$ consists of $5$ fixed points,
 $\Fix(\sigma_{5})$ is the union of a curve $C_2$
of genus $2$ and one point and $\Fix(\sigma_{3})$ is the union of a genus $2$ curve $C'_2$ and $2$ points.
Let $\varphi:X\to \P^2$ be the morphism associated to the linear system $|C'_2|$,
which is a degree two morphism branched along a plane sextic $B$ which 
possibly contracts the smooth rational curves disjoint from $C'_2$  to simple singular points of $B$ \cite{SD}.
Since $[C'_2]$ is fixed by $\sigma_{15}^*$,  the automorphism $\sigma_{15}$ 
descends to an automorphism $\bar\sigma_{15}$ of $\P^2$. 
Let $\bar\sigma_3=\bar\sigma_{15}^5$ and  $\bar\sigma_5=\bar\sigma_{15}^3$.
Up to a projectivity we can assume that $\bar\sigma_{15}$, and thus $\bar\sigma_3$ and $\bar\sigma_5$ are diagonal.
Observe that both $\bar\sigma_3$ and $\bar\sigma_5$ must fix pointwise a line and a point in $\P^2$,
since both $\sigma_3$ and $\sigma_5$ fix pointwise a curve of positive genus.
Moreover, by the previous description, the two lines must be distinct. 
Thus we can assume that 
\[
\bar\sigma_3(x_0:x_1:x_2)=(x_0:x_1:\zeta_3x_2)\quad 
\bar\sigma_5(x_0:x_1:x_2)=(x_0:\zeta_5x_1:x_2).
\]
The branch sextic $B$ of $\varphi$ is invariant for $\bar\sigma_{15}$.
Observe that $B$ can not contain a line fixed by either $\bar\sigma_3$ or $\bar\sigma_5$ 
since otherwise $\Fix(\sigma_3)$ and $\Fix(\sigma_5)$ would contain a smooth rational curve.
This implies that $B$ is defined by an equation of the form $F_6(x_0,x_1,x_2)=0$ 
as given in Example \ref{ex15a}.
If either $a_2$ or $a_3$ vanishes, then $B$ would contain a line.
If $a_1=0$, then $B$ would contain a singular point of type $E_8$, 
whose central component would be fixed by both $\sigma_3$ and $\sigma_5$, a contradiction.
Thus $a_1a_2a_3\not=0$, in particular $B$ is smooth. 
Up to rescaling the variables, an equation for $X$ is the one given 
in Table \ref{1dim} with $a\in \C$.

{\em Case b)}. We now consider the case when $\Fix(\sigma_{15})$ consists of $7$ points, 
$\Fix(\sigma_5)$ is the union of an elliptic curve $C_1$ and $4$ points and 
$\Fix(\sigma_3)$ is the union of a genus $4$ curve $C_4$, a rational curve $R$ and one point.
Let $\pi:X\to \P^1$ be the morphism associated to the linear system $|C_1|$, which
is an elliptic fibration.  Since $C_1$ is invariant for $\sigma_{15}$, then the elliptic fibration 
is invariant for $\sigma_{15}$. 
On the other hand, since $\sigma_3$ fixes the curve $C_4$ of genus $g>1$, then 
it induces the identity on $\P^1$. 
The automorphism $\sigma_3$ acts on $C_1$ and has fixed points in $C_1\cap C_4$, 
thus it has exactly $3$ fixed points by the Riemann-Hurwitz formula.
This implies that $C_1\cdot C_4\leq 3$. Moreover, $C_1\cdot C_4>1$ since otherwise 
the restriction of $\pi$ to $C_4$ would be an isomorphism onto $\P^1$.
Thus $C_1\cdot C_4$ is either $2$ or $3$.
We now show that the second case does not appear. 

If $C_1\cdot C_4=3$, then the curve $R$ can not intersect $C_1$, since otherwise 
$C_1$ would contain more than three fixed points of $\sigma_5$.
Thus $R$ must be contained in a reducible fiber $F$ of $\pi$. 
The fiber $F$ is invariant for $\sigma_3$, it can only contain 
an isolated fixed point of $\sigma_3$ and $C_4\cdot F=3$.
By \cite[Lemma 4.1]{ArtebaniSarti} $F$ should be of type $I_0^*$,
but this contradicts the fact that the rank of the invariant lattice of $\sigma_3$ 
is $4$ by \cite[Theorem 2.2]{ArtebaniSarti}.

Thus $C_1\cdot C_4=2$. Since the general fiber of $\pi$ must contain 
three fixed points of $\sigma_3$ by the Riemann-Hurwitz formula,
then the curve $R$ must be a section of $\pi$.
Thus $\pi$ is a jacobian elliptic fibration invariant for an order three automorphism 
and with a fixed section. This implies that up to a coordinate change $\pi$ has 
Weierstrass equation 
\[
y^2=x^3+p(t),
\]
with $\sigma_3(x,y,t)=(\zeta_3x,y,t)$, where $\deg(p)\leq  12$. 
In these coordinates $C_1$ is the fiber over $t=0$ and $C_4$ is the curve $x=0$.
Since $\sigma_5$ has order $5$ on $R$, then it induces an order five automorphism on $\P^1$ 
which can be assumed to be $\bar\sigma_5(t)=\zeta_5t$.
Since $\sigma_5$ is the identity when $t=0$, then $\sigma_5(x,y,t)=(x,y,\zeta_5t)$.
Thus up to a coordinate change $\pi$ has Weierstrass equation of the form
\[
y^2=x^3+(t^5-1)(t^5-a),
\]
with $a\in \C$ and $\sigma_{15}(x,y,t)=(\zeta_3x,y,\zeta_5t)$,
as in Example \ref{ex15b}.

{\em Case c)}. Finally, assume that $\Fix(\sigma_{15})$ consists of $4$ points,
$\Fix(\sigma_5)$ is the union of an elliptic curve $C_1$ and $4$ points and 
$\Fix(\sigma_3)$ is a curve of genus $4$.
Following the same argument in the proof of Lemma \ref{g50} 
we obtain that the quotient of $X$ by $\sigma_3$ is 
a smooth rational surface $Y$ isomorphic to a Hirzebruch surface $\mathbb F_r$ for some even $r\geq 0$. 
This quotient is a cyclic degree three covering branched along a smooth curve $B$ of genus four whose class 
is $[B]=\frac{3(r+2)}{2}f+3e$, where $f^2=0, e^2=-r, f\cdot e=1$.
Since $B$ is smooth, then $[B]\cdot e=\frac{6-3r}{2}\geq 0$, thus $r\leq 2$.

The case when $r=0$, i.e. $Y\cong \P^1\times\P^1$ can be excluded as follows.
In this case $B$ is a curve of type $(3,3)$.
The automorphism $\sigma_5$ descends to an automorphism  $\bar\sigma_5$ of $\P^1\times \P^1$ 
which has a fixed curve, thus up to a coordinate change we can assume
\[
\bar\sigma_5:(x_0:x_1),(y_0:y_1)\mapsto (x_0:x_1),(\zeta_5y_0:y_1).
\]
However there is no curve of type $(3,3)$ which is invariant for this automorphism,
giving a contradiction.

Thus $Y\cong \mathbb F_2$, $[B]=6f+3e$ and $[B]\cdot e=0$.
After contracting the $(-2)$-curve of $Y$ we obtain the surface $P\cong \P(1,1,2)$. 
Since the contraction is a crepant morphism the image of $B$ is a smooth curve $D$ with 
$2[D]\sim -3K_P$, i.e. of degree $6$. 
The automorphism $\sigma_5$ descends to an automorphism  $\bar\sigma_5$  
of $P$ which preserves $D$ and fixes the image of the curve $C_1$.
This implies that, after a coordinate change, we can assume $\bar\sigma_5(x_0,x_1,x_2)=(\zeta_5x_0,x_1,x_2)$.
The equation of $D$ must be invariant for $\bar\sigma_5$, thus it is of the form given in Example \ref{ex15c}. 
If all the coefficients of $G_6$ are non-zero, then one obtains the equation in Table \ref{1dim} up to rescaling the variables.
\end{proof}

\subsection{Order 30}

\begin{proposition}\label{prop30}
Let $X$ be a K3 surface with a purely non-symplectic automorphism $\sigma_{30}$ of order $30$
such that $\dim(V^\sigma)=2$.
Then there are two possibilities for the fixed locus of $\sigma_{30}$ and of its powers:
\[
\begin{array}{c||c|c|c|c|c}
& \Fix(\sigma_{30}) & \Fix(\sigma_{15}) & \Fix(\sigma_{5})&\Fix(\sigma_{3})&\Fix(\sigma_{2})\\
\hline
a) & \{p_1\} &  \{p_1,\ldots, p_5\} & C_{2}\sqcup \{p_1\}& C'_2\sqcup\{p_2,p_3\} &C_{10}\\
b) & \{p_1, p_2,p_5\}&\{p_1,\ldots,p_7\}&C_1\sqcup\{p_1,\ldots,p_4\}&C_4\sqcup R\sqcup\{p_1\}&C_9\sqcup R\\
\end{array}
\]
where $C_g, C_g'$ have genus $g$ and $g(R)=0$.
Moreover, $d=(2,0,1,0,0,0,1,1)$ in case a) and $d=(2,0,0,1,0,0,1,3)$ in case b).

Finally ${\rm NS}(X)\cong U(3)\oplus A_2\oplus A_2$ for a very general K3 surface $X$ in case a) 
and ${\rm NS}(X)\cong H_5\oplus A_4$ for a very general K3 surface $X$ in case b). 
\end{proposition}

\begin{proof}
Let $\chi_{i}=\chi(\Fix(\sigma_{i})), i=30,15,5,3,2$.
First observe that given a one dimensional family of K3 surfaces admitting a purely non-symplectic automorphism of order 30, 
every element in the family admits a purely non-symplectic automorphism of order 15. 
Thus this corresponds to one of the three families in Proposition \ref{prop15} 
and the vector $(\chi_{15},\chi_5,\chi_3)$ is either 
$(5,-1,0)$, $(7,4,-3)$ or $(4,4,-6)$.

Decomposing $H^2(X,\C)$ as the direct sum of the eigenspaces of $\sigma^*$ we obtain:
\begin{equation}
22=8d_{30}+8d_{15}+4d_{10}+2d_{6}+4d_5+2d_3+d_2+d_1.
\label{30}
\end{equation}
Assuming $d_{30}=2$, this gives $d_{15}=0$.
Using the topological Lefschetz fixed point formulas we compute the topological Euler characteristic 
of the fixed loci of powers of $\sigma_{30}$ by:
\begin{equation}
\begin{cases}
\chi_{30}&=d_{10}+d_6-d_5-d_3-d_2+d_1\\
\chi_{15}&=-d_{10}-d_6-d_{5}-d_{3}+d_2+d_1+4\\
\chi_5&=-d_{10}+2d_6-d_5+2d_3+d_2+d_1-2\\
\chi_3&=4d_{10}-d_6+4d_5-d_3+d_2+d_1-6\\
\chi_2&=-4d_{10}-2d_6+4d_5+2d_3-d_2+d_1-14.
\end{cases}\label{sistema30}
\end{equation}

We first assume to be in case a) of Proposition \ref{prop15}, i.e. 
$\Fix(\sigma_{5})$ is the union 
of a smooth curve $C_2$ of genus $2$ and a point $p_1$,
$\Fix(\sigma_{3})$ is the union of 
a smooth curve $C'_2$ of genus $2$ and two isolated points 
and $\Fix(\sigma_{15})$ consists of $5$ isolated points $p_1,\dots,p_5$.
In particular $(\chi_{15},\chi_5,\chi_3)=(5,-1,0)$.
Moreover, since the fixed locus of $\sigma_{15}$ only contains isolated points,
the same holds for $\sigma_{30}$. Thus $\chi_{30}\geq 0$.
By \eqref{30} and \eqref{sistema30} we get the possibilities in Table \ref{tab:30}.
{\small
\begin{table}[h]
\begin{tabular}{cccccccc|ccccc}
 $d_{30}$&$d_{15}$&$d_{10}$&$d_{6}$&$d_{5}$&$d_{3}$&$d_{2}$&$d_1$&$\chi_{30}$&$\chi_{15}$&$\chi_5$&$\chi_3$&$\chi_2$\\
 \hline
 \rowcolor{Lgray}    2& 0& 1& 0& 0& 0& 1& 1& 1& 5& -1& 0& -18 \\
     2& 0& 1& 0& 0& 0& 0& 2& 3& 5& -1& 0& -16 \\
     2& 0& 0& 0& 1& 0& 0& 2& 1& 5& -1& 0& -8    
     \end{tabular}
     \vspace{0.2cm}
\caption{}
\label{tab:30}
\end{table}}

In particular $\chi_{30}$ is either $3$ or $1$,
thus 
$\Fix(\sigma_{30})$ is either the union of $p_1$ and two of the $p_i$'s with $i\geq 2$ (and the other two are exchanged)
or $\Fix(\sigma_{30})=\{p_1\}$ and $\sigma_{30}$ has no fixed points on $C_2$.
By the proof of Theorem \ref{main} in the case $n=15$,
the linear system associated to $C_2$ 
defines a double cover $\varphi:X\to \P^2$ which can be defined in $\P(1,1,1,3)$ by an equation of the form
\[
y^2=x_0^6+x_0x_1^5+x_2^6+ax_0^3x_2^3,
\]
where $a\in \C$ and in these coordinates
$
\sigma_{15}(x_0,x_1,x_2,y)=(x_0,\zeta_5x_1,\zeta_3x_2,y).
$
Since $\sigma_{30}$ preserves $C_2$, then it induces an 
automorphism $\bar\sigma_{30}$ of $\P^2$.
The involution $\sigma_2$ either induces the identity or an involution of $\P^2$.
The latter is not possible since the fixed locus of $\sigma_2$ would 
contain a curve of genus at most $2$, while $\chi_2\leq -8$ by Table \ref{tab:30}.
 Thus $\sigma_2$ coincides with the  (automorphism induced by) 
the covering involution of $\varphi$, which fixes a smooth genus $10$ curve, so that $\chi_2=-18$ 
and  $\chi_{30}=1$ by Table \ref{tab:30}. Thus $\sigma_{30}$ fixes a unique point.
Since $C_2$ is invariant for $\sigma_{30}$, then $\varphi(C_2)$ is a line 
which contains two fixed points for $\bar \sigma_{30}$.
Since $\chi_{30}=1$,  their preimages by $\varphi$ 
are four points exchanged in pairs by $\sigma_2$.

Assume now to be in case b) of Proposition \ref{prop15}, i.e. 
$\Fix(\sigma_{3})$ is the union of a curve $C_4$ of genus $4$, 
a rational curve $R$ and a point, 
$\Fix(\sigma_{5})$ is union of an elliptic curve $C_1$ and four points, 
and $\Fix(\sigma_{15})$ is the union of $7$ points (3 on $C_1$).
In particular $(\chi_{15},\chi_5,\chi_3)=(7,4,-3)$.
Moreover, $\chi_{30}\geq 0$ since $\Fix(\sigma_{15})$ only contains 
isolated points, and thus the same holds for $\sigma_{30}$.
There are five possible vectors $d$ such that $(\chi_{15},\chi_5,\chi_3)=(7,4,-3)$
(see Table \ref{tab:30-1}). 
{\small
\begin{table}[h]
\begin{tabular}{cccccccc|ccccc}
 $d_{30}$&$d_{15}$&$d_{10}$&$d_{6}$&$d_{5}$&$d_{3}$&$d_{2}$&$d_1$&$\chi_{30}$&$\chi_{15}$&$\chi_5$&$\chi_3$&$\chi_2$\\
 \hline
     2& 0& 0& 1& 0& 0& 2& 2& 1& 7& 4& -3& -16 \\
   \rowcolor{Lgray}2& 0& 0& 1& 0& 0& 1& 3& 3& 7& 4& -3& -14 \\
     2& 0& 0& 0& 0& 1& 1& 3& 1& 7& 4& -3& -10 \\
  2& 0& 0& 1& 0& 0& 0& 4& 5& 7& 4& -3& -12 \\
    2& 0& 0& 0& 0& 1& 0& 4& 3& 7& 4& -3& -8 \\
\end{tabular}
   \vspace{0.2cm}
\caption{}
\label{tab:30-1}
\end{table}}
By the proof of Theorem \ref{main} in the case $n=15$,
 $X$ admits an elliptic fibration $\pi:X\to \P^1$ with Weierstrass equation
\[
y^2=x^3+(t^5-1)(t^5-a),
\]
with $a\in \C$ and $\sigma_{15}(x,y,t)=(\zeta_3x,y,\zeta_5t)$.
By the same argument in the proof of Theorem \ref{main} in the case $n=15$,
using  \cite[Lemma 5]{ArtebaniSarti4}, one concludes that the elliptic fibration 
is invariant for $\sigma_{30}$. Since $\chi_2\leq -8$, then $\sigma_2$ 
fixes a curve of genus $>1$. Such curve is clearly transverse to all fibers 
of $\pi$, thus $\sigma_2$ induces the identity on the basis of the fibration.
Moreover, $\sigma_2$ must fix the section at infinity $R$ of the fibration, 
since it preserves $R$ and  each fiber of $\pi$.
This implies that $\sigma_2(x,y,t)=(x,-y,t)$.
In particular $\sigma_2$ fixes $R$ and the curve defined by $y=0$,
which has genus $9$, so that $\chi_2=-14$.
Moreover $\sigma_{30}$ fixes three points: 
two points on $R$ and the center of the fiber of type $IV$ over $t=\infty$.

Assume now to be in case c) of Proposition \ref{prop15}, i.e. 
$\Fix(\sigma_{3})$ is  a curve $C_4$ of genus $4$, 
$\Fix(\sigma_{5})$ is union of an elliptic curve $C_1$ and four points, 
and $\Fix(\sigma_{15})$ is the union of $4$ points.
In particular $(\chi_{15},\chi_5,\chi_3)=(4,4,-6)$.
Moreover, $\chi_{30}\geq 0$ since $\Fix(\sigma_{15})$ only contains 
isolated points, and thus the same holds for $\sigma_{30}$.
By the proof of Theorem \ref{main} in the case $n=15$, 
$X$ is the minimal resolution of the double cover of $P=\P(1,1,2)$
branched along a smooth curve $D$ of degree $6$ not passing 
through the singular point of $P$. 
The automorphism induced by $\sigma_5$ in $P$ can be assumed to 
be $(x_0,x_1,x_2)\mapsto (\zeta_5 x_0,x_1,x_2)$.
The automorphism $\sigma_2$, since it commutes with $\sigma_3$,  
induces an involution $\bar\sigma_2$ of $P$ which preserves the curve $D$. 
Moreover it can be also diagonalized. However, no diagonal involution leaves 
invariant the general equation as in Example \ref{ex15c}, thus this case 
is not possible.

The N\'eron-Severi lattice of a very general $X$ in cases a) and b)   is clearly the same as in Proposition \ref{prop15}.
\end{proof}

\begin{example}\label{ex30a}
The double cover $X$ of $\P^2$ in Example \ref{ex15a}
carries the order $30$ automorphism 
\[
\sigma_{30}(x_0,x_1,x_2,x_3)=(x_0,\zeta_5 x_1,\zeta_3x_2,-x_3).
\]
Observe that for a general choice of the coefficients the fixed locus of $\sigma_2$ is the smooth plane sextic defined by 
$x_3=0$, which has genus $10$. Moreover, the fixed locus of $\sigma_{30}$ consists 
of the point $(0,1,0,0)$. This is an example of case a) in Proposition \ref{prop30}.
\end{example}

\begin{example}\label{ex30b}
The elliptic K3 surface in Example \ref{ex15b}
carries the order $30$ automorphism
\[
\sigma_{30}(x,y,t)=(\zeta_3 x, -y,\zeta_5 t).
\]
Observe that for general $a\in \C$ the fixed locus of $\sigma_2$ 
is the curve $y=0$, which has genus $9$.
Moreover, as observed in the proof of Proposition \ref{prop30}, 
the fixed locus of $\sigma_{30}$ consists of two points in the section at infinity (over $t=0$ and $t=\infty$) 
and the center of the fiber of type $IV$ over $t=\infty$.  This is an example of case b) in Proposition \ref{prop30}.
\end{example}

\begin{proof}[Proof of Theorem \ref{main}, order 30]
Let $X$ be a K3 surface with a purely non-symplectic automorphism $\sigma$ of order $30$ 
such that $\dim(V^\sigma)=2$. 
It is straightforward from the proof of Proposition \ref{prop30} that, up to isomorphism, 
$(X,\sigma)$ belongs to one of 
the families in Examples \ref{ex30a} and \ref{ex30b}.
\end{proof}

\subsection{Order 16}
Purely non-symplectic automorphisms of order 16 
on K3 surfaces have been classified in \cite{AlTabbaaSartiTaki}.
The following result has the same statement as \cite[Theorem 4.1]{AlTabbaaSartiTaki},
but we provide a slightly different proof since we use the weaker hypothesis 
$\dim(V^\sigma)=2$.

\begin{proposition}\label{teo-16}
Let $\sigma_{16}$ be a purely non-symplectic automorphism of order $16$ 
of a K3 surface $X$ and assume that 
$\dim(V^\sigma)=2$ (or equivalently $S(\sigma_2)$ has rank $6$). 
Then there are two possibilities for the fixed locus of $\sigma_{16}$  
and of its powers:
{\small 
\[
\begin{array}{c||c|c|c|c}
& \Fix(\sigma_{16}) & \Fix(\sigma_{8}) & \Fix(\sigma_{4})&\Fix(\sigma_{2})\\
\hline
a) & \{p_1,\dots,p_6\}\sqcup R & \{p_1,\dots,p_6\}\sqcup R  &  \{p_1,\dots,p_6\}\sqcup R &  C_7\sqcup R\sqcup R'  \\
b) & \{p_1,p_2, p_7,p_8\}&  \{p_1,\dots,p_6\}\sqcup R&   \{p_1,\dots,p_6\}\sqcup R  & C_6\sqcup R 
\end{array}
\]
}
where $g(C_6)=6$, $g(C_7)=7$ and $g(R)=g(R')=0$.
Moreover, $d=(2,0,0,0,6)$ in case a) and $d=(2,0,0,2,4)$ in case b).

Finally ${\rm NS}(X)\cong U\oplus D_4$ for a very general $X$ in case a) and ${\rm NS}(X)\cong U(2)\oplus D_4$ 
for a very general $X$ in case b).
\end{proposition}

 \begin{proof}
Decomposing $H^2(X,\C)$ as the direct sum of the eigenspaces 
of $\sigma_{16}^*$ we obtain:
\begin{equation}
22=8d_{16}+4d_{8}+2d_{4}+d_{2}+d_1.
\label{16}
\end{equation}
Since $d_{16}=2$, 
this implies that $d_{8}$ is either 0 or 1 and gives the 14 possibilities for the vector $d$ in Table \ref{tab16}.
{\small \begin{table}[h]
\begin{tabular}{ccccc|cccc}
$d_{16}$&$d_8$&$d_4$&$d_2$&$d_1$&$\chi_{16}$&$\chi_8$&$\chi_4$&$\chi_2$\\
\hline
 2& 1& 0& 1& 1& 2& 4& 0& -8 \\
     2& 0& 2& 1& 1& 2& 2& 8& -8 \\
     2& 0& 1& 3& 1& 0& 5& 8& -8 \\
     2& 0& 0& 5& 1& -2& 8& 8& -8 \\
     2& 1& 0& 0& 2& 4& 4& 0& -8 \\
     2& 0& 2& 0& 2& 4& 2& 8& -8 \\
     2& 0& 1& 2& 2& 2& 5& 8& -8 \\
     2& 0& 0& 4& 2& 0& 8& 8& -8 \\
     2& 0& 1& 1& 3& 4& 5& 8& -8 \\
     2& 0& 0& 3& 3& 2& 8& 8& -8 \\
     2& 0& 1& 0& 4& 6& 5& 8& -8 \\
   \rowcolor{Lgray}   2& 0& 0& 2& 4& 4& 8& 8& -8 \\
     2& 0& 0& 1& 5& 6& 8& 8& -8 \\
  \rowcolor{Lgray}    2& 0& 0& 0& 6& 8& 8& 8& -8 
     \end{tabular}
        \vspace{0.2cm}
     \caption{}
\label{tab16}
\end{table}}

Let $N_i$ be the number of isolated fixed points of $\sigma_i$, $\chi_i=\chi(\Fix(\sigma_{i}))$
and 
\[
\alpha_i=\sum_{C\subset\Fix(\sigma_i)}(1-g(C))
\] 
for $i\in\{2,4,8,16\}$.
By the topological Lefschetz fixed point 
formula we get
 \begin{equation}
\begin{cases}
\chi_{16}=&=-d_2+d_1+2\\
\chi_{8}=&=-d_{4}+d_2+d_1+2\\
\chi_4=&=-4d_8+2d_4+d_2+d_1+2\\
\chi_2=&=-8d_{16}+4d_8+2d_4+d_2+d_1+2.
\end{cases}\label{sistema16}
\end{equation}
Table \ref{tab16}  shows the values of $(\chi_{16},\chi_8,\chi_4,\chi_2)$ for each possible vector $d$.

Observe that $\chi_2=-8$. By \cite{Nikulin}, $N_2=0$ and
$\Fix(\sigma_2)$ is the union of a curve of genus $g$ and $k$ 
rational curves with $(g,k)=(5,0),(6,1)$ or $(7,2)$.

Moreover, $\chi_4=0$ or $8$.
By \cite[Proposition 1]{ArtebaniSarti4} we have that $N_4=2\alpha_4+4$. 
Since $\chi_4=2\alpha_4+N_4$ one has 
\[
\chi_4=4\alpha_4+4.
\]
If $\chi_4=0$, then $\alpha_4=-1$, but this is not possible since $\Fix(\sigma_4)\subseteq\Fix(\sigma_2)$ and it is not compatible with the aforementioned possibilities for $\Fix(\sigma_2)$.
Thus $\chi_4=8$, $\alpha_4=1$ and $\Fix(\sigma_4)$ contains a rational curve (and no more curves) and 6 points.
This implies that  the case $(g,k)=(5,0)$ is impossible.

The cases $(g,k)=(6,1)$ and $(7,2)$ are treated in Lemma \ref{lemma1} and Lemma \ref{lemma2}. 
We conclude that the only admissible cases are the ones in Proposition \ref{teo-16}.
Observe that both in case a) and b) we have that $d_4=d_8=0$ and $d_1+d_2=6$.
This implies that $S(\sigma_8)=S(\sigma_4)=S(\sigma_2)$ has rank $6$. 
If $X$ is very general, then $\rk\,{\rm NS}(X)=22-2\varphi(20)=6$ and thus by Remark \ref{ns} 
${\rm NS}(X)=S(\sigma_2)$. 
Moreover by \cite[Theorem 4.2.2]{Nikulin-inv} or \cite[Figure 1]{ArtebaniSartiTaki} 
 the invariant lattice of $\sigma_2$ is isometric to $U\oplus D_4$
 in case a) and $U(2)\oplus D_4$ in case b), see \cite{AlTabbaaSartiTaki}.
\end{proof} 

\begin{lemma}\label{lemma1}
If $Fix(\sigma_2)$ is the union of a curve of genus 6 and a rational curve, then the fixed loci of $\sigma_{16},\sigma_8$ and $\sigma_4$ are as follows:
\[\small
\Fix(\sigma_{16})= \{p_1,p_2, p_7,p_8\},\ \Fix(\sigma_8)=\{p_1,\dots,p_6\}\sqcup R,\] 
\[
\Fix(\sigma_4)=\{p_1,\dots,p_6\}\sqcup R. \]
\end{lemma}

\begin{proof} Let $C_6$ (resp. $R$) be the smooth curve  of genus $6$ (resp. rational curve) in $\Fix(\sigma_2)$.
By the previous analysis we know that $\sigma_4$ fixes pointwise $R$ and 
has $6$ isolated fixed points $p_1,\dots,p_6$ on $C_6$.
 
By the Riemann-Hurwitz formula for $\sigma_{8}$ on $C_6$, 
we observe that either a) $2$ of the $p_i$'s are fixed and the other four are permuted in pairs 
by $\sigma_8$ or b) the points $p_1,\dots,p_6$ are fixed points for $\sigma_8$.
Observe that case a) is not possible since $\chi_4=8$ and $\chi_8=4$ does not appear 
in Table \ref{tab16}.

 By the Riemann-Hurwitz formula for $\sigma_{16}$ on $C_6$, we obtain that
$\sigma_{16}$ fixes $2$ of the $p_i$'s and exchanges the other four  
 in pairs. Thus $(\chi_{16},\chi_8,\chi_4,\chi_2)=(4,8,8,-8)$.

Observe that six of the fixed points of $\sigma_8$ lie on a curve fixed pointwise by $\sigma_2$ 
and not by $\sigma_4$, thus the local action of $\sigma_8$ at such points is  
either of type $A_{2,8}$ or $A_{3,8}$. 
By \cite[Proposition 2.2]{TabbaaSarti} we have that $6=2+4\alpha_8$, thus $\alpha_8=1$.
This implies that $N_8=6$ and the curve $R$ is pointwise fixed by $\sigma_8$.
On the other hand, by \cite[Proposition 2]{AlTabbaaSartiTaki}, 
$N_{16}$ is bigger or equal to $2\alpha_{16}+1$. This implies that $\alpha_{16}=0$,
i.e. $R$ is not pointwise fixed by $\sigma_{16}$.
 \end{proof}

\begin{lemma}\label{lemma2}
If $Fix(\sigma_2)$ is the union of a curve of genus 7 and 2 rational curves, then the fixed loci of $\sigma_{16},\sigma_8$ and $\sigma_4$ are as follows:
\small \[\Fix(\sigma_{16})=\{p_1,\dots,p_4,q_1,q_2\}\sqcup R,\ \Fix(\sigma_8)= \{p_1,\dots,p_4,q_1,q_2\}\sqcup R,\]
\[ \Fix(\sigma_4)= \{p_1,\dots,p_4,q_1,q_2\}\sqcup R.
\]
\end{lemma}

\begin{proof} 
Let $C_7$ (resp. $R,R'$) be the smooth curve of genus $7$ (resp. rational curves) in $\Fix(\sigma_2)$. 
We already know that one rational curve is fixed by $\sigma_4$, say $R$. 
Thus $\sigma_4$ fixes $2$ points $q_1,q_2$ on $R'$ and $4$ points $p_1,\ldots,p_4$ on $C_7$.
This implies that the curves $R$ and $R'$ cannot be exchanged by $\sigma_{16}$ nor by $\sigma_8$
and that $\chi_{16}\geq4$ and $\chi_8\geq4$.

By the Riemann-Hurwitz formula for $\sigma_{8}$ on $C_7$,
either the four $p_i$'s are fixed by $\sigma_8$ or none of them is fixed by $\sigma_8$. 
This implies that either $\chi_8=4$ or $\chi_8=8$.
Looking at Table \ref{tab16}, we find that we are left with the three possibilities of Table \ref{tab16-2}.
\begin{table}[h!]
\begin{tabular}{ccccc|cccc}
$d_{16}$&$d_8$&$d_4$&$d_2$&$d_1$&$\chi_{16}$&$\chi_8$&$\chi_4$&$\chi_2$\\
\hline
2& 0& 0& 2& 4& 4& 8& 8& -8 \\
2& 0& 0& 1& 5& 6& 8& 8& -8 \\
2& 0& 0& 0& 6& 8& 8& 8& -8 
     \end{tabular}
        \vspace{0.2cm}
     \caption{}
\label{tab16-2}
\end{table}

In particular $\chi_{8}=8$ and $\{p_1,\dots, p_4,q_1,q_2\}\subset\Fix(\sigma_8)$.
Moreover, by \cite[Proposition 2.2]{TabbaaSarti} we obtain that $2+4\alpha_8=6$, thus $\alpha_8=1$.
This implies that $\sigma_8$ fixes pointwise the curve $R$.

By the Riemann-Hurwitz formula for $\sigma_{16}$ on $C_7$, 
either a) $\sigma_{16}$ fixes the four $p_i$'s and thus $\chi_{16}=8$, 
or b) it does not fix any of them and $\chi_{16}=4$.
By \cite[Proposition 2, Remark 1.3]{AlTabbaaSartiTaki} the cases $(N_{16},\alpha_{16})=(2,1)$ 
and $(N_{16},\alpha_{16})=(8,0)$ are impossible.
Thus in case a) $\alpha_{16}=1$ and $N_{16}=4$, i.e. the fixed locus of 
$\sigma_{16}$ contains $p_1,\dots,p_4,q_1,q_2$ 
and the curve $R$.
On the other hand in case b) we have that $\alpha_{16}=0$, i.e. 
$\sigma_{16}$ fixes exactly $q_1,q_2$ and two points on $R$.
We now show that this case can not appear.
By \cite[Remark 1.3]{AlTabbaaSartiTaki} if $N_{16}=4$, then $n_{3,16}=n_{7,16}=1$ and
$n_{8,16}=2$. Observe that the  points of type $A_{8,16}$ lie on a curve fixed by $\sigma_8$,
thus they must be the two points on $R$. 
This implies that the points of type $A_{3,16}$ and $A_{7,16}$ are $q_1, q_2$.
However, two isolated fixed points of $\sigma_{16}$ lying on an invariant smooth rational curve  
can not be of these types by the proof of \cite[Lemma 4]{ArtebaniSarti4}.
 \end{proof}

For the following examples, see \cite[Example 4.2]{AlTabbaaSartiTaki}.

\begin{example}\label{ex16-1}
Consider the elliptic fibration defined by
\[
y^2=x^3+t^2x+at^3+t^{11},\ a\in \C
\]
with the order $16$ automorphism $\sigma_{16}(x,y,t)=(\zeta_8x,\zeta_{16}^3y,\zeta_8t)$.
The action of $\sigma_{16}^*$ on the holomorphic two form $\omega_X=(dx \wedge dt)/2y$ is the multiplication by $\zeta_{16}$,
thus $\sigma_{16}$ is purely non-symplectic.
The fibration has a fiber of type $I_0^*$ over $t=0$ and a fiber of type $II$ over $t=\infty$. 
The automorphism $\sigma_{16}$ fixes the central component of the fiber of type ${I^*_0}$, four points in the other components 
of the same fiber and two more fixed points in the fiber over $t=\infty$.
This is an example of case a) in Proposition \ref{teo-16}.
\end{example}

\begin{example}\label{ex16-2} Consider the plane sextic $B$ defined by 
\[
F_6(x_0,x_1,x_2)=x_0(x_0^4x_2+a_1x_1^5+a_2x_1x_2^4+a_3x_1^3x_2^2)=0
\]
for general $a_1,a_2,a_3\in \C$. Observe that  $B$ is the union of a smooth plane quintic $C$ and a line $L$. 
Let $Y$ be the double cover of $\P^2$  branched along $B$, which can be 
defined by the equation $
x_3^2=F_6(x_0,x_1,x_2)
$
in $\P(1,1,1,3)$.
The surface $Y$ has the order $16$ automorphism
\[
\sigma_{16}(x_0,x_1,x_2,y)=(x_0,\zeta_8^7x_1,\zeta_8^3x_2,\zeta_{16}^3y).
\]
The surface $Y$ has $5$ singular points of type $A_1$ over the intersection points 
of $C$ and $L$. Its minimal resolution $X$ is a K3 surface and $\sigma_{16}$ lifts to an 
automorphism $\tilde\sigma_{16}$ of $X$. The automorphism $\tilde\sigma_{16}$ has $4$ fixed points: 
two of them over the points $(1,0,0,0)$ and $(0,1,0,0)$ and the other two in the exceptional divisor 
over $(0,0,1,0)$ (which is a singular point of $Y$). Thus this is an  example of case $b)$ in Proposition 
\ref{teo-16}.
\end{example}

\begin{proof}[Proof of Theorem \ref{main}, order 16]

Let $X$ be a K3 surface with a purely non-symplectic automorphism $\sigma$ of order $16$ 
such that $\dim(V^\sigma)=2$.  
By Proposition \ref{teo-16}, $\Fix(\sigma_{16})$ is either the union of a rational curve and $6$ points 
or the union of $4$ isolated points. 

{\em Case a).}   
By Proposition \ref{teo-16} ${\rm NS}(X)=S(\sigma_2)\cong U\oplus D_4$
  for a very general K3 surface.
In what follows we assume $X$ to be very general. 
By \cite[Lemma 2.1]{kondo_trivial} or \cite[\S 3, proof of Corollary 3]{shapiro}, $X$ has an elliptic fibration 
$
\pi:X\rightarrow \P^1
$
with a section $S$ and a reducible fiber of type $\tilde D_4=I_0^*$.
The curve $C_7$ fixed by the involution $\sigma_2$ has to be transverse to the fibers of $\pi$, 
since its genus is bigger than $1$. 
Thus $\sigma_2$ induces the identity on the basis of the fibration.
Since ${\rm NS}(X)= S(\sigma_2)$, $\sigma_2^*$ is the identity on ${\rm NS}(X)$,
hence each smooth rational curve is invariant for $\sigma_2$. 
This implies that the section $S$  and the central component of the fiber of type $I_0^*$ 
are pointwise fixed by $\sigma_2$. 
Since a smooth fiber of $\pi$ must contain four fixed points for $\sigma_2$ and one of them is on $S$, 
then $C_7$ intersects it in $3$ points. 
Applying \cite[Lemma 5]{ArtebaniSarti4} with $x=[C_7]$, one concludes that the elliptic fibration 
$\pi$ is invariant under $\sigma_{16}$.
The section $S$ corresponds to the curve $R'$ (see notation of Proposition \ref{teo-16}) i.e. it is not fixed pointwise by $\sigma_{16}$, 
otherwise each fiber of $\pi$, including the smooth ones, would have an order $16$ automorphism with a fixed point,  
which is impossible for an elliptic curve \cite[]{Hartshorne}. 
Thus $\sigma_{16}$ induces an automorphism of order $8$ on the basis of $\pi$.
This implies that $\sigma_8$ preserves each fiber of $\pi$ and acts on it as an involution with a fixed point. 

Consider a Weierstrass equation for $\pi$ with respect to the section $S$: 
\[
y^2=x^3+A(t)x+B(t),\quad t\in\P^1.
\]
We can assume that the two invariant fibers of $\sigma_{16}$ are over $t=0$ and $t=\infty$,
and that the fiber of type $I_0^*$ is over $t=0$.
Since ${\rm NS}(X)\cong U\oplus D_4$  the fiber $F_{0}$ over $0$ is the only reducible fiber of $\pi$,
moreover $24-e(F_0)$ is divisible by $8$. This implies that the fiber over $t=\infty$ is of type $II$.
By \cite[Table IV.3.1]{Miranda}, this implies that the vanishing order $v(\Delta)$ of $\Delta(t)$ at $t=0$ is $6$ and at $t=\infty$ is $2$. 
Thus $\Delta(t)=t^6P(t)$ with $P(0)\neq0, \deg(P(t))=16$.
Moreover $v(B(\infty))=1$, thus $B(t)=t^3Q(t)$ with $\deg(Q(t))=8$. 
Since the action of $\sigma_{16}$ on the basis of $\pi$ 
has order 8 and the fibers over $t=0,\infty$ are preserved by $\sigma_{16}$, then $Q(t)=t^8+a$ 
with $a\in\C$. 
By \cite[Table IV.3.1]{Miranda}, we have that $A(t)=t^2$. 
Moreover, $\sigma_{16}(x,y,t)=(\zeta_8x,\zeta_{16}^3y,\zeta_8t)$ and $X$ belongs to the 
family in Example \ref{ex16-1}.

{\em Case b).} In this case $\Fix(\sigma_{16})$ is the union of $4$ isolated points,
and $S(\sigma_2)\cong U(2)\oplus D_4$ by Proposition \ref{teo-16}. As before we assume $X$ to be very general, i.e. that 
${\rm NS}(X)=S(\sigma_2)$. It is known that the surface $X$ has a degree two morphism 
 $\pi:X\to \P^2$ which is the minimal resolution of a double cover ramified along the union of 
 a line $\ell$ and a quintic curve $C$ \cite[Section 4]{AlTabbaaSartiTaki}.
In particular, $X$ has six $(-2)$-curves, i.e. $5$ exceptional divisors $E_1,\dots, E_5$ over the points $\ell\cap C$ 
and the proper transform $E$ of (the double cover) of $\ell$.
It follows from Vinberg's algorithm (see \cite{roulleau}) that these are the only $(-2)$-curves of $X$.
This implies that the linear system of the divisor $2E+\sum_{i=1}^5E_i$, which is the one defining the morphism $\pi$,
 is invariant for $\sigma^*$. Thus the automorphism $\sigma$ induces an automorphism $\bar\sigma$ of $\P^2$ preserving the 
 branch curve $\ell\cup C$.

The involution $\sigma_2$ fixes a genus 6 curve and a rational curve $R$.
If the induced automorphism $\overline{\sigma}_2$ were an involution, 
it would fix a line and a point. 
This would imply that in $\Fix(\sigma_2)$ the maximum possible genus is 2, giving a contradiction. 
Thus $\overline{\sigma}_2$ is the identity on $\P^2$, 
$\sigma_2$ is the covering involution and 
$R$ is the transform of the line $\ell$ in the branch locus.
Moreover, since $R$ is contained in $\Fix(\sigma_4)$ and $\Fix(\sigma_8)$, 
$\ell$ is fixed by $\overline{\sigma}_4$ and $\overline{\sigma}_8$.
In addition, $\overline{\sigma}_{16}$ has order 8 on $\P^2$ and it only fixes points, 
$\overline{\sigma}_{8}$ has order 4 and 
$\overline{\sigma}_{4}$ has order 2. 

Assume that $\ell$ is defined by $x_0=0$, thus 
\[
\overline{\sigma}_{8}(x_0,x_1,x_2)=(ix_0,x_1,x_2), \quad 
\overline{\sigma}_{4}(x_0,x_1,x_2)=(-x_0,x_1,x_2).
\]
Since $\overline{\sigma}_{16}$ only fixes points in $\P^2$,
it is $\overline{\sigma}_{16}(x_0,x_1,x_2)=(\zeta_8x_0,x_1,-x_2)$
and the equation of $X$ is obtained taking invariant monomials
and recalling that we need the quintic to be smooth, 
otherwise $\Fix(\sigma_2)$ would not contain a genus $6$ curve.
Thus the equation of $X$ is as in Example \ref{ex16-2}
and 
\[
\sigma_{16}(x_0,x_1,x_2,y)=(\zeta_8x_0, x_1,-x_2,\zeta_{16}y)=(x_0,\zeta_8^7x_1,\zeta_8^3x_2,\zeta_{16}^{11}y),
\]
where $\zeta_{16}$ is a primitive $16$-th root of unity with $\zeta_{16}^2=\zeta_8$.
Observe that $\sigma_{16}^9$ is the automorphism in Example \ref{ex16-2}.
If all the coefficients of $F_6$ are non-zero, then one obtains the equation in Table \ref{1dim} up to rescaling the variables.
\end{proof}

\subsection{Order 20}

\begin{proposition}\label{prop20}
Let $X$ be a K3 surface with a purely non-symplectic automorphism $\sigma_{20}$ of order $20$ 
such that $\dim(V^\sigma)=2$.
Then the fixed loci of $\sigma_{20}$ and of its powers are as follows:
\[
\begin{array}{c|c|c|c|c}
\Fix(\sigma_{20}) & \Fix(\sigma_{10}) & \Fix(\sigma_{5})&\Fix(\sigma_{4})&\Fix(\sigma_{2})\\
\hline
 \{p_1,p_2,p_3\} & \{p_1,\dots,p_7\} & C_2\sqcup\{p_1\} & \{p_1,\dots,p_6\}\sqcup R &  C_6\sqcup R
\end{array}
\]
where $g(C_i)=i$ for $i=2,6$ and $g(R)=0$.
Moreover $d=(2,0,1,0,0,2)$ and ${\rm NS}(X)=S(\sigma_2)$ for a very general such K3 surface $X$.%
\end{proposition}

\begin{proof}
Decomposing $H^2(X,\C)$ as the direct sum of the eigenspaces 
of $\sigma_{20}^*$ we obtain:
 \[22=8d_{20}+4d_{10}+4d_5+2d_4+d_2+d_1.\]
Since $d_{20}=2$, then $d_{10}$ is either $0$ or $1$ and this gives $16$ possibilities for the vector $d$.
Let $\chi_i=\chi(\Fix(\sigma_{i})), i\in\{2,4,5,10,20\}$.
By the topological Lefschetz fixed point 
formula we get
\begin{equation}
\begin{cases}
\chi_{20}&=d_{10}-d_5-d_2+d_1+2\\
\chi_{10}&= 2d_{20}-d_{10}-d_5-2d_4+d_2+d_1+2\\
\chi_{5}&=-2d_{20}-d_{10}-d_5+2d_4+d_2+d_1+2\\
\chi_{4}&= -4d_{10}+4d_5-d_2+d_1+2\\
\chi_{2}&= -8d_{20}+4d_{10}+4d_5-2d_4+d_2+d_1+2\\
\end{cases}\label{sist20}
\end{equation}

By \eqref{sist20} we compute $\chi_5$ for the $16$ possible $d$'s 
and find that it is either $-1$ or $4$.
  Lemma \ref{lemma3} and Lemma \ref{lemma4} study these two cases separately.
  Observe that, since $d_{20}=2$ and $\varphi(20)=8$, 
the N\'eron-Severi lattice of a very general $X$ has rank $22-2\cdot 8=6$.
Moreover $S(\sigma_2)\subseteq {\rm NS}(X)$ by Remark \ref{ns}.
On the other hand, since the fixed locus of $\sigma_2$ is the union of a curve of genus $6$ 
and a rational curve, then ${\rm rk}\,S(\sigma_2)=6$ by \cite{Nikulin}, thus $S(\sigma_2)={\rm NS}(X)$.
 \end{proof}

 \begin{lemma}\label{lemma3}
 If $\chi_5=-1$, then the fixed loci of $\sigma_{20}, \sigma_{10}, \sigma_5, \sigma_4, \sigma_2$ are
 \[ \Fix(\sigma_{20})=\{p_1,p_2,p_3\},\ \Fix(\sigma_{10})=\{p_1,\dots,p_6, p\}\]
 \[ \Fix(\sigma_5)= C_2\sqcup\{p\}, \ \Fix(\sigma_4)=\{p_1,\dots,p_6\}\sqcup R, \Fix(\sigma_2)= C_6\sqcup R
 \]
 \end{lemma}
 
 \begin{proof} 
 By \cite{ArtebaniSartiTaki}, if $\chi_5=-1$, then $\Fix(\sigma_5)$ is the union of a smooth curve $C$
 of genus two and one point  $p$. 
 This corresponds to the cases in Table \ref{tab:20}.
 {\small
\begin{table}[h!]
\begin{tabular}{cccccc|ccccccc}
 $d_{20}$&$d_{10}$&$d_{5}$&$d_{4}$&$d_{2}$&$d_1$&$\chi_{20}$&$\chi_{10}$&$\chi_5$&$\chi_4$&$\chi_2$\\
 \hline
 \rowcolor{Lgray}  2 & 0 & 1 & 0 & 0 & 2 & 3 & 7 & -1 & 8 & -8\\
 2 & 0 & 1 & 0 & 1 & 1 & 1 & 7 & -1 & 6 & -8\\
 2 & 1 & 0 & 0 & 0 & 2 & 5 & 7 & -1 & 0 & -8\\
 2 & 1 & 0 & 0 & 1 & 1 & 3 & 7 & -1 & -2& -8\\
\end{tabular}
\vspace{0.2cm}

\caption{}
\label{tab:20}
\end{table}}
In all these cases $\chi_{10}=7$,
 so that $\Fix(\sigma_{10})$ is the union of $p$ and $6$ points on $C$.
 By the Riemann-Hurwitz formula, this implies that $\sigma_{20}$ has two fixed points on $C$,
 so that $\Fix(\sigma_{20})$ consist of the union of three points.
 Moreover, in all these cases $\chi_2=-8$ by Table \ref{tab:20},
 so that by \cite{Nikulin} $\Fix(\sigma_2)$ is a) the union of a curve of genus $7$ and two rational curves,
 or b) the union of a curve of genus six and a rational curve, or c) a genus $5$ curve.
 In all cases $\sigma_{20}$ acts with order $10$ on the curve of positive genus,
since otherwise either $\sigma_{10}$ or $\sigma_4$ should contain such curve in its fixed locus,
contradicting the previous remarks for $\sigma_{10}$ and \cite[Theorem 0.1]{ArtebaniSarti4}.

 In case a), $\sigma_{20}$ must fix exactly three points on the curve $C$ of genus $7$ 
 and exchange the two rational curves.  By the Riemann-Hurwitz formula this implies that $\sigma_5$ 
 fixes the same points on $C$ and  $\sigma_4$ has exactly $8$ fixed points on $C$ and exchanges 
 the two rational curves.
 This is not possible since by \cite[Theorem 0.1]{ArtebaniSarti4} the number of fixed points equals $2\alpha+4$,
 where $\alpha=\sum_{C\subset \Fix(\sigma_4)}(1-g(C))$.

 In case b), $\sigma_{20}$  has exactly one fixed point on the genus six curve 
and two points on the rational curve $R$, while $\sigma_5$ has exactly five fixed points on the curve by 
the Riemann-Hurwitz formula. By the same formula $\sigma_4$ has $6$ fixed points on $C$.
 By \cite[Theorem 0.1]{ArtebaniSarti4} $R$ is fixed by $\sigma_4$. This case corresponds to the statement.
 
Case c) is impossible since, by the Riemann-Hurwitz formula, a curve of genus $5$ can not 
 have an order five automorphism with more than two fixed points (and $\sigma_5$ would 
 have this property).
 \end{proof}
 
 \begin{lemma}
\label{lemma4}
If $\chi_5=4$, there are no admissible cases.
 \end{lemma}

 \begin{proof} 
By \cite{ArtebaniSartiTaki}, if $\chi_5=4$, then 
$\Fix(\sigma_5)$ contains either four isolated points 
or an elliptic curve and four isolated points. 
In both cases $a_{1,5}=3, a_{2,5}=1$.
Observe that 
points of type $A_{4,20}, A_{5,20}, A_{9,20}$ lie on a curved fixed by $\sigma_5$ while
points of type $A_{i,20}$ with $i\in\{1,2,3,6,7,8\}$ are isolated fixed points for $\sigma_5$.
Since the action of $\sigma_{20}$ on $\Fix(\sigma_5)$ has order $2$ or $4$,  
in both cases the point of type $A_{2,5}$ is fixed by $\sigma_{20}$ and $a_{2,20}+a_{7,20}=1$ and $a_{1,20}+a_{3,20}+a_{6,20}+a_{8,20}$ is either $1$ or $3$.

If $\Fix(\sigma_5)$ consists of four isolated points, 
$a_{4,20}+a_{5,20}+a_{9,20}=0$ since there are no curves in $\Fix(\sigma_5)$.
A Magma computation shows that the
holomorphic Lefschetz formula has no solutions satisfying these conditions.

If $\Fix(\sigma_5)$ consists of four isolated points and an elliptic curve $E$,
by the Riemann-Hurwitz formula $E$ contains $0$, $2$ or $4$ isolated points for $\sigma_{20}$ 
thus
$a_{4,20}+a_{5,20}+a_{9,20}\in\{0,2,4\}$. 
 The holomorphic Lefschetz formula 
 has no solutions with these restrictions.
\end{proof}

\begin{example}\label{ex20} 
 Let $B$ be the plane sextic defined by 
 \[
 F_6(x_0,x_1,x_2)=x_0(x_1^5+a_1x_2^5+a_2x_0^2x_2^3+a_3x_0^4x_2)=0,
 \]
 where $a_1,a_2,a_3\in\C$ are general. 
Observe that $B$ is the union of a smooth plane quintic $C$ and a line $L$.
Let $Y$ be the double cover of $\P^2$
 branched along $B$, which  can be defined by the equation 
 $x_3^2-F_6(x_0,x_1,x_2)=0$ in $\P(1,1,1,3)$.
The surface $Y$ has the order $20$ automorphism
\[
\sigma_{20}(x_0,x_1,x_2,y)=(-x_0,\zeta_{5}x_1,x_2,iy).
\]
The surface $Y$ has $5$ singular points of type $A_1$ over the intersection points 
of $C$ and $L$. Its minimal resolution $X$ is a K3 surface and $\sigma_{20}$ lifts to an 
automorphism $\tilde\sigma_{20}$ of $X$. The automorphism $\tilde\sigma_{20}$ has $3$ fixed points 
over $(1,0,0), (0,1,0), (0,0,1)$. 
\end{example}

\begin{proof}[Proof of Theorem \ref{main}, order 20]
Let $X$ be a K3 surface with a purely non-symplectic automorphism $\sigma_{20}$ of order $20$.
By Proposition \ref{prop20} 
$\Fix(\sigma_{5})$ contains a curve $C_2$ of genus $2$ and one point.
The linear system $|C_2|$ defines a morphism $\varphi:X\to \P^2$ 
of degree two which contracts all smooth rational curves 
orthogonal to $C_2$. 
Since $\sigma$ leaves $C_2$ invariant, then it induces an 
automorphism $\bar\sigma$ of $\P^2$ which can be assumed 
to be  diagonal.

Let $\bar\sigma_2=\bar\sigma^{10}$ and assume it has order two.
Thus its fixed locus is the union of a line and one point,
so that $\Fix(\sigma_2)$ contains a fixed curve of genus at most $2$,
contradicting the fact that $\sigma_2$ fixes a curve of genus $6$.
Thus $\sigma_2$ coincides with the covering involution of $\varphi$.

Now consider the automorphism $\bar\sigma_4=\bar\sigma^5$,
whose order is equal to two. Its fixed locus contains a line, we can assume it to be 
$L=\{x_0=0\}$ up to projectivities. By Proposition \ref{prop20} the line $L$ 
must be a component of the branch curve $B$ of $\varphi$.

Finally, let $\bar\sigma_5=\bar\sigma^4$. Since $\sigma_5$ 
has a fixed curve, then $\bar\sigma_5$ must fix a line $L'$ which 
is not equal to the line fixed by $\bar\sigma_4$,
thus up to projectivities we can assume $L'=\{x_1=0\}$.
In these coordinates 
\[
\bar\sigma(x_0,x_1,x_2)=(-x_0,\zeta_5x_1,x_2).
\]
The branch curve $B$ is reduced, invariant for $\bar\sigma$ and 
must contain the line $L$ as a component. This implies that 
its equation is  as in Example \ref{ex20}. 
If all the coefficients of $F_6$ are non-zero, then 
one obtains the equation in Table \ref{1dim} up to rescaling the variables.
 \end{proof}

\begin{remark} It follows from the proof of Theorem \ref{main}, $n=20$,
that there are five smooth rational  curves $R_1,\dots,R_5$ in $X$, each intersecting at one point the 
two fixed curves $C_6$ and $R$ of $\sigma_2$. 
The classes of the curves $R, R_1,\dots, R_5$ all belong to the invariant lattice $S(\sigma_2)$.
Observe that the classes of 
\[2R+R_1+R_2+R_3+R_4,\ 2R+R_1+R_2+R_3+R_5,\ 
R,\ R_1,\ R_2,\ R_3,
\]
 generate a lattice $S$ isometric to $U(2)\oplus D_4$.
Since $S$ is contained in $S(\sigma_2)$ and $\det(S(\sigma_2))=\det(S)=-2^4$ by \cite[Theorem 0.1]{ArtebaniSartiTaki}, then $S=S(\sigma_2)$.
\end{remark}

\subsection{Order 24}

\begin{proposition}\label{prop24}
Let $X$ be a K3 surface with a purely non-symplectic automorphism $\sigma=\sigma_{24}$ of order $24$ 
such that $\dim(V^\sigma)=2$. Then the fixed loci of $\sigma_{24}$ and some of its powers are 
as follows:
{\small
\[
\begin{array}{c|c|c}
\Fix(\sigma_{24}) & \Fix(\sigma_{12}) & \Fix(\sigma_{6})\\
\hline
 \{p_1, p_2, p_3,p_{12},p_{13}\} & \{p_1, p_2, p_3,p_{12},p_{13}\} & R_1\sqcup\{p_1,\dots,p_{11}\}
\end{array}
\]
\[
\begin{array}{c|c}
\Fix(\sigma_{3})&\Fix(\sigma_{2})\\
\hline
C_4\sqcup R_1\sqcup\{p_1\} &  C_7\sqcup R_1\sqcup R_2
\end{array}
\]
}where $g(C_i)=i$ for $i=4,7$ and $g(R_1)=g(R_2)=0$. 
Moreover $\chi(\Fix(\sigma_4))=\chi(\Fix(\sigma_8))=8$,
$d=(2,0,0,0,0,1,0,4)$ and ${\rm NS}(X)=S(\sigma_2)\cong U\oplus D_4$ for a very general such K3 surface $X$.
\end{proposition}

\begin{proof}
Decomposing $H^2(X,\C)$ as the direct sum of the eigenspaces 
of $\sigma_{24}^*$ we obtain:
\[22=8d_{24}+4d_{12}+4d_8+2d_6+2d_4+2d_3+d_2+d_1.\]
Let $\chi_i=\chi(\Fix(\sigma_{i})), i\in\{2,3,4,6,8,12,24\}$.
 By the topological Lefschetz fixed point 
formula we get

\begin{equation}
\begin{cases}
\chi_{24}&=d_6-d_3-d_2+d_1+2\\
\chi_{12}&=2d_{12}-d_6-2d_4-d_3+d_2+d_1+2\\
\chi_{8}&=-2d_6+2d_3-d_2+d_1 +2\\
\chi_{6}&=4d_{24}-2d_{12}-4d_8-d_6+2d_4-d_3+d_2+d_1 +2\\
\chi_{4}&= -4d_{12}+2d_6-2d_4+2d_3+d_2+d_1 +2\\
\chi_{3}&=-4d_{24}-2d_{12}+4d_8-d_6+2d_4-d_3+d_2+d_1+2\\
\chi_{2}&= -8d_{24}+4d_{12}-4d_8+2d_6+2d_4+2d_3+d_2+d_1 +2\\
\end{cases}
\end{equation}

Computing all possible values of the vector $d$ one can see 
that $\chi_3\in \{0,-3,-6\}$. 

Assume $\chi_3=0$. By \cite{ArtebaniSarti} $\Fix(\sigma_3)$ 
is either the union of  genus two curve and two isolated points or the union of a genus three curve,
a smooth rational curve and two isolated points.
Clearly $\Fix(\sigma_6)\subseteq \Fix(\sigma_3)$ and in this case $\chi_6=16$ or $8$. 
The first case is incompatible with the structure of the fixed locus of $\sigma_3$.
If $\chi_6=8$, then the fixed locus of $\sigma_3$ must be  the union of a genus three curve $C$,
a smooth rational curve $R$ and two isolated points $p,q$.
The automorphism $\sigma_6$ fixes $4$ points on $C$ and $p,q$. Moreover, it either fixes pointwise $R$ 
or it has two isolated fixed points on it.
Both cases are incompatible with \cite[Theorem 4.1]{Dillies} since the fixed points of $\sigma_6$ 
contained in the fixed curve of $\sigma_3$ are those of type $A_{2,6}$ (of type $\frac{1}{6}(3,4)$ in \cite{Dillies}).

If $\chi_3=-6$ we have $\chi_6=10$ and this 
can be seen to be incompatible with \cite[Theorem 4.1]{Dillies}  
with an argument similar to the previous one.

If $\chi_3=-3$ and by \cite{ArtebaniSarti} $\Fix(\sigma_3)$ 
is either the union of a curve of genus $3$ and one point,
or the union of a curve of genus $4$, a smooth rational curve and one point.
In these cases we have $\chi_6=13$, 
which excludes the first possibility for $\Fix(\sigma_3)$.
Thus $\Fix(\sigma_3)$ is the union of a curve $C$ of genus $4$, 
a smooth rational curve $R$ and one point $p$.
Using the Riemann-Hurwitz formula for $\sigma_6$ and the fact that 
$\chi_6=13$ we obtain that $\sigma_6$ fixes $p$ and
$10$ points on $C$. Moreover, by  \cite[Theorem 4.1]{Dillies} the curve $R$ 
is pointwise fixed by $\sigma_6$.
 In this case one computes that $\chi_{12}$ is either $5$ or $1$,
 but the second case is not possible since $\sigma_{12}$ either fixes pointwise 
 or has two fixed points on $R$.
 Thus $\Fix(\sigma_{12})$ fixed $p$, two points on $C$ 
 and it either fixes pointwise or has two fixed points on $R$.
 A computation using the holomorphic Lefschetz formula shows that 
 the first case does not occur.
 In this case one computes that $\chi_{24}\in \{-1,1,3,5,7\}$.
 The only cases compatible with the structure of $\Fix(\sigma_{12})$ are $\chi_{24}=3$ or $5$.
 The first case is impossible by the Riemann-Hurwitz formula.
 
 Assuming $\chi_3=-3$, $\chi_6=13$, $\chi_{12}=5$ 
 and $\chi_{24}=5$ we find two possible vectors $d=(2,0,0,0,0,1,0,4), (2,0,0,1,0,0,1,3)$.
 For these cases $\chi_2=-8$, $\chi_4=8$. 
 Moreover $\chi_8=8$ in the first case and $2$ in the second case.
 
 By \cite{Nikulin} the fixed locus of $\sigma_2$ is either the union of a curve $C_7$ of genus $7$ 
 and two smooth rational curves ($R_1$ and $R_2$), or the union of a curve $C_6$ of genus $6$ and $R_1$. 
 The latter is not possible by the Riemann-Hurwitz formula applied to $\sigma_6$ restricted to $C_6$.
 Since $\chi_4=8$, $\sigma_4$ must fix $4$ points on $C_7$, 
 two points on  $R_1$ and it either fixes pointwisely $R_2$ or it has two fixed points on it.
 This implies that $\Fix(\sigma_8)$ contains isolated points and, at most, a smooth rational curve.
 Thus $\chi_8\geq \chi_{24}=5$, which excludes the case $d=(2,0,0,1,0,0,1,3)$.
 
 Finally, by \cite[Theorem 4.2.2]{Nikulin-inv} or \cite[Figure 1]{ArtebaniSartiTaki} 
 the invariant lattice of $\sigma_2$ is isometric to $U\oplus D_4$.
 For a very general K3 surface we have ${\rm rk}\, {\rm NS}(X)=22-2\varphi(24)=6$.
 Moreover $S(\sigma_2)\subseteq {\rm NS}(X)$ by Remark \ref{ns}, 
 thus  $S(\sigma_2)= {\rm NS}(X)$.
 \end{proof}
 
 \begin{example}\label{ex24} 
  Consider the elliptic surface with equation
  \[
  y^2=x^3+t(t^4-1)(t^4-a),\ a\in \C.
  \]
  For general $a\in \C$ it is a K3 surface and carries the order $24$ automorphism
  \[
  \sigma_{24}(x,y,t)=(\zeta_{12}x,\zeta_8y,it).
  \]
  The action of $\sigma^*$ on the holomorphic two form $\omega_X=(dx \wedge dt)/2y$ is the multiplication by $\zeta_{12}\zeta_4\zeta_8^{-1}$,
thus $\sigma_{24}$ is purely non-symplectic.
For general $a\in \C$ the elliptic fibration has a singular fiber $F_{\infty}$ of type $I_0^*$ over $t=\infty$ and $9$ fibers of type $II$.
The automorphism $\sigma_2$ fixes the section at infinity $R_1$, the genus $7$ curve defined by $y=0$ and the central 
component $R_2$ of the fiber $F_{\infty}$.
The automorphism $\sigma_3$ fixes $R_1$, the curve of genus  $4$ defined by $x=0$ and the intersection point $p_1$ 
between $R_2$ and the component of $F_{\infty}$ intersecting $R_1$. Observe that the remaining three components 
of $F_{\infty}$ are permuted by $\sigma_3$.
The automorphism $\sigma_{6}$ fixes the $9$ singular points $p_3,\dots,p_{11}$ of the fibers of type $II$,
the point $p_1$ and the intersection point $p_2$ between the fiber $F_{\infty}$ and the curve $x=0$.
Finally, the automorphisms $\sigma_{12}$ and $\sigma_{24}$ fix the singular point $p_3$ of the fiber $F_0$ of type $II$ over $t=0$,
the intersection points of $R_1$ with the fibers $F_0, F_{\infty}$, $p_1$ and $p_2$.  
\end{example}

 \begin{proof}[Proof of Theorem \ref{main}, order 24]
 By Proposition \ref{prop24} the fixed locus of $\sigma_2$ is the union of a curve $C_7$ of genus $7$ and two rational curves $R_1, R_2$.
Moreover ${\rm NS}(X)=S(\sigma_2)\cong U\oplus D_4$ for a very general  $X$. 
Following the first part of the proof of Theorem \ref{main} for order $16$,
 we find that a very general $X$ has a jacobian elliptic fibration $\pi:X\to \P^1$ with a fiber of type $I_0^*$ 
 such that $R_1$ can be assumed to be a section of $\pi$, 
 $R_2$ is the central component of the reducible fiber, and $C_7$ 
 intersects a general fiber in three points.
 It follows from \cite[Lemma 5]{ArtebaniSarti4} with $x=[C_7]$ that $\pi$ is invariant for $\sigma_{24}$.
 Since the fixed locus of $\sigma_3$ contains a curve $C_4$ of genus $>1$, then 
 each fiber of $\pi$ is invariant for $\sigma_3$. Moreover $\sigma_3$ fixes pointwise the curve $R_1$.
 Thus up to a coordinate change $\pi$ has Weierstrass equation  of the form
\[
y^2=x^3+p(t),
\]
where $\deg(p)\leq 12$, $\sigma_2(x,y,t)=(x,-y,t)$ and $\sigma_3(x,y,t)=(\zeta_3x,y,t)$.
 Observe that $\sigma_8$ preserves $R_1$ but $\sigma_4=\sigma_8^2$ does not fix 
 it pointwisely,  since otherwise $R_1$ would be contained in the fixed locus of $\sigma_{12}$,
 contradicting Proposition \ref{prop24}.
Thus $\sigma_8$ induces an automorphism $\bar\sigma_8$ of order $4$ on $\P^1$.
 Up to a coordinate change, we can assume that $\bar\sigma_8(t)=it$. 
Since the reducible fiber of type $I_0^*$ must be preserved by $\sigma_8$,
 then we we can assume it to be over $t=\infty$. 
 By \cite{Miranda} this implies that the $\deg(p)=9$ (so that it has a triple root at infinity).
 Moreover, since its zero set is invariant for $\bar\sigma_8$, then $p(t)=t(t^4-a)(t^4-b)$ for some 
 $a,b\in \C$. Finally, since $\sigma_8^4=\sigma_2$ we can assume that
 $ \sigma_8(x,y,t)=(-ix,\zeta_8y,it).$ This implies that up to a coordinate change 
 $X$ belongs to the family in Example \ref{ex24}.
 \end{proof}

\section{Classification for order 22}\label{class22}
We now provide a classification theorem of 
purely non-symplectic automorphisms $\sigma$ of order $22$ on a K3 surface
according to their fixed locus.  
Observe that, since $\varphi(22)=10$, then $\dim(V^\sigma)\in \{1,2\}$. 
The case when $\dim(V^\sigma)=2$ has  been studied in Section \ref{s2}.

We recall that the fixed locus of any power $\sigma_i:=\sigma^\frac{22}{i}$  of $\sigma$ is of the form
\[
 C_i\sqcup R_1\sqcup \ldots\sqcup R_{k_i}\sqcup \{p_1,\dots, p_{N_i}\},
\]
where $g(C_i)=g_i$, $g(R_\ell)=0$ for $\ell=1,\dots,k_i$.

\begin{remark}
As in \cite[Lemma 1.3]{ACV}, a straightforward computation using the holomorphic Lefschetz formula shows 
that a non-symplectic automorphism of order $22$ is purely non-symplectic.
\end{remark}
\begin{theorem}\label{thm22}
Let $\sigma$ be a purely non-symplectic automorphism of order $22$ of a complex K3 surface $X$.
Then the invariants $(g_i,k_i,N_i)$ of the fixed locus of $\sigma_i:=\sigma^\frac{22}{i}$, the vector 
$d=(d_{22,}d_{11}, d_2,d_1)$ giving the dimensions of the eigenspaces of $\sigma^*$ in $H^2(X,\mathbb C)$ 
and the N\'eron-Severi lattice of a very general K3 surface carrying an automorphism with such invariants 
are given by one of the rows of Table \ref{tab22}. Moreover, all cases in the table exist.

\begin{table}[h!!!]
\centering
\begin{tabular}{c|ccc|ccc|cc|c|c}
&$N_{22}$ & $g_{22}$& $k_{22} $  & $N_{11}$& $g_{11}$ & $k_{11}$  & $g_2$ & $k_2$ & $d$ & ${\rm NS}$ \\
\hline
A1&6&-& 0& 2 &  1& 0&  10 & 1& (2,0,0,2) & $U$ \\
B1&11&0& 0& 11& 0 & 0 & 5 & 5 & (1,0,1,11) & $U\oplus A_{10}$ \\
B2&9&0& 0& 11& 0 & 0 & 5 & 4 & (1,0,2,10)& $U\oplus A_{10}$\\
B3&5&-& 0& 11& 0 & 0& 5 & 1 & (1,0,5,7) & $U\oplus A_{10}$\\
\end{tabular}
 \vspace{0.2cm}
\caption{Order 22}
\label{tab22}
\end{table}
\end{theorem}

\begin{proof}
Let $\sigma_{11}$ be the square of $\sigma_{22}$.
According to \cite[Table 4]{ArtebaniSartiTaki}, the fixed locus of $\sigma_{11}$ is 
either a) the union of a smooth elliptic curve and 2 points, 
or b) the union of a rational curve and 11 points.
In the first case $m:=\frac{1}{10}(22-\rk S(\sigma_{11}))=2$, while in the second case $m=1$.

Recall that fixed points of type $A_{10,22}$ lie on a curve in $\Fix(\sigma_{11})$, 
while points of type $A_{i,22}, A_{10-i,22}$ 
correspond to isolated points for $\sigma_{11}$ of type $A_{i,11}$, $i=1,\ldots,4$.
The Lefschetz holomorphic formula with the restrictions 
\[
a_{5,22}\leq a_{5,11},\quad a_{i,22}+a_{10-i,22}\leq a_{i,11},\quad i=1,2,3,4
\]
gives the solutions as in Table \ref{tabB}, 
where we compute $\chi_{22}$ and $\chi_{2}$ by \eqref{sistema22}.
{\begin{table}[h]
\begin{tabular}{c|cc|cc}
&$(a_{1,22}, a_{2,22},\ldots, a_{10,22})$&$\alpha$&$\chi_{22}$&$\chi_2$\\
\hline
A1&(0, 0, 0, 1, 0, 0, 0, 0, 1, 4)& 0& 6& -16  \\
B1&(3, 2, 1, 1, 1, 2, 1, 0, 0, 0)&1&13&2\\
B2&(3, 2, 2, 1, 1, 0, 0, 0, 0, 0)&1&11&0\\
B3&(0, 0, 0, 1, 1, 0, 0, 0, 1, 2)&0&5&-6\\
\end{tabular}
   \vspace{0.2cm}
\caption{}
\label{tabB}
\end{table}}

In case A1 we have $d_{1}+d_{2}=2$ and $d_{11}+d_{22}=2$ by \cite[Table 4]{ArtebaniSartiTaki}.
Since $\chi_{22}=6$, then $(d_{22},d_{11},d_2,d_1)=(2,0,0,2)$ by \eqref{sistema22}.
The description of the fixed locus of $\sigma_2$ is thus obtained as in the proof of Proposition \ref{prop22}.

We now study the possibilities for $\Fix(\sigma_{22})$ when
$\Fix(\sigma_{11})$ is the union of a rational curve and $11$ points. 
By \cite{Nikulin} the fixed locus of the involution $\sigma_2$
is the union of a curve of genus $g_2$ and $k_2$ rational curves and 
$\chi_2=2(1-g_2+k_2)$. 
Thus in case B1 one has $(g_2,k_2)\in\{(0,0), (1,1),(2,2),(3,3),(4,4), (5,5)\}$.
The only admissible one is $(g_2,k_2)=(5,5)$, since otherwise, recalling that isolated points of $\sigma_{22}$ lie on fixed curves for $\sigma_2$,
one gets a contradiction with the Riemann-Hurwitz formula.

As for case B2, one has $(g_2,k_2)\in\{(1,0), (2,1),(3,2),(4,3),(5,4), (6,5)\}$.
The first four cases give a contradiction to the Riemann Hurwitz formula. 
Case $(g_2,k_2)=(6,5)$ is not admissible since by \cite[Proposition V.2.14]{FK} a curve of genus 6 does not admit an automorphism of order 11 acting on it.

Similarly, in case B3 the possibilities are $(g_2,k_2)\in\{(4,0), (5,1),(6,2)\}$ and the only admissible one is $(g_2,k_2)=(5,1)$.
The vector $d=(d_{22},d_{11},d_2,d_1)$ is obtained in all cases by means of \eqref{sistema22}.

The N\'eron-Severi lattice of a very general K3 surface in case A1 has been given in Section \ref{sec:22}.
In the remaining cases, which have $d_{22}=1$, the rank of the N\'eron-Severi lattice in the very general case is $22-\varphi(22)=12$.
Since the lattice $S(\sigma_{11})$ is a primitive sublattice of ${\rm NS}(X)$ by Remark \ref{ns} and has rank $d_1+d_2=12$ in each case, then 
${\rm NS}(X)=S(\sigma_{11})$. By \cite{OZ2} or \cite[Sec. 7]{ArtebaniSartiTaki} the lattice $S(\sigma_{11})$ is isometric to $U\oplus A_{10}$.

An example for case A1 has been given in Section \ref{sec:22}.
We now provide examples for  the cases B1, B2, B3.
\end{proof}

 \begin{example}\label{B1} (Case B1)
Let $X$ be the elliptic K3 surface whose elliptic fibration is given by
\[y^2=x^3+t^7x+t^5.
\]
The singular fibers of the fibration are $II^*$ over $t=0$, $III$ over $t=\infty$ and 11 fibers of type $I_1$.
The automorphism 
\[\sigma_{22}:(x,y,t)\mapsto(\zeta_{22}^2x,\zeta_{22}^3y,\zeta_{22}^{10}t)\]
is purely non-symplectic of order $22$ since its action on the two form $\frac{dx\wedge dt}{2y}$ is the multiplication 
by $-\zeta_{11}$. The automorphism $\sigma_{22}$ preserves the fibers over $t=0$ and $t=\infty$.
In the fiber over $t=0$, which is of type $II^*$, it must fix the component of multiplicity $6$ and 
has $8$ isolated fixed points in the other components. 
In the fiber over $t=\infty$ it fixes three isolated points.
The involution $\sigma_2$ preserves each fiber of the elliptic fibration, thus
it must fix $R$, three more components of the fiber over $t=0$, the section at infinity and 
the $3$-section $y=0$, which has genus $5$. 
This corresponds to case B1.
\end{example}

\begin{example}\label{B2} (Case B2) Let us consider the elliptic fibration
\[y^2=x^3+t^5x+t^2,\]
The fibration has a fiber of type $IV$ over $t=0$, a fiber of type $III^*$ over $t=\infty$ 
and $11$ fibers of type $I_1$.
The automorphism
\[
\sigma_{22}(x,y,t)=(\zeta_{11}^8x,-\zeta_{11}y,\zeta_{11}t)
\]
is purely non-symplectic of order $22$ since its action on the two form $\frac{dx\wedge dt}{2y}$ is the multiplication 
by $-\zeta_{11}^8$. 
By \cite[Example 7.4]{ArtebaniSartiTaki}, $\sigma_{11}$ has fixed locus $R\cup\{p_1,\ldots,p_{11}\}$, 
where $R$ is the central component of the fiber of type $III^*$.
The involution $\sigma_2$ maps $(x,y,t)$ to $(x,-y,t)$, 
thus it preserves each fiber. This implies that it
fixes $R$ and two more rational components of the fiber of type $III^*$, 
as well as the section at infinity and the $3$-section $y=0$, whose genus is $5$.
This corresponds to case B2.
 \end{example}

\begin{example}\label{B3} (Case B3) We already observed in Section \ref{sec:22} that 
the elliptic K3 surface  defined by
\[y^2=x^3+ax+(t^{11}-1), \ a\in\C^*
\]
with the automorphism  $\sigma_{22}:(x,y,t)\mapsto (x:-y:\zeta_{11}t)$, is an example of case A.
If  $a^3=-\frac{27}4$, 
thus the fibration admits a singular fiber of type $II$ over $t=0$, $I_{11}$ over $t=\infty$ and $11$ fibers of type $I_1$.
The  fixed locus of the automorphism $\sigma_{11}$ is contained in the fibers over $t=0$ and $t=\infty$. 
Since it fixes $11$ isolated points and one rational curve, then it must fix one of the components of the fibre of type $I_{11}$, say $R$,
has $9$ fixed points in the other components of the same fibers and two more fixed points in the fiber of type $II$.
The involution $\sigma_2$ fixes the section at infinity and the curve $y=0$, which has genus $5$.
Moreover, $\sigma_2$ can not preserve each component of the fiber of type $I_{11}$ by \cite[Lemma 4]{ArtebaniSarti4}.
Thus $\sigma_2$ acts on the fiber of type $I_{11}$ as a reflection, without fixed components 
and with a unique invariant component.
This corresponds to case B3.
\end{example}

\begin{remark} 
By Section \ref{moduli} the moduli space of K3 surfaces having a purely non-symplectic automorphism of order $22$ 
whose invariants are as in cases B1, B2 or B3  is $0$-dimensional,  since $\dim(V^\sigma)=1$. 
In fact, since ${\rm rk}\,T(X)=10=\varphi(22)$ and $f^*$ has order $11$ on ${\rm NS}(X)$, then  
it follows from \cite[Theorem 5.9]{brandhorst} that there is a unique K3 surface $X$ which carries the three types 
of non-symplectic automorphisms of order $22$.  Thus the K3 surfaces given in Examples \ref{B1}, \ref{B2} and \ref{B3} 
are isomorphic. \end{remark}

\section{Classification for order 15}\label{class15}
We now provide a classification theorem of 
purely non-symplectic automorphisms $\sigma$ of order $15$ on a K3 surface
according to their fixed locus.  
Observe that, since $\varphi(15)=8$, then $\dim(V^\sigma)\in \{1,2\}$. 
The case when $\dim(V^\sigma)=2$ has  been studied in Section \ref{s2}.

\begin{theorem}\label{thm15}
Let $\sigma$ be a purely non-symplectic automorphism of order $15$ of a complex K3 surface $X$.
Then the invariants $(g_i,k_i,N_i)$ of the fixed locus of $\sigma_i:=\sigma^\frac{15}{i}$, the vector 
$d=(d_{15,}d_{5}, d_3,d_1)$ giving the dimensions of the eigenspaces of $\sigma^*$ in $H^2(X,\mathbb C)$ 
and the N\'eron-Severi lattice of a very general K3 surface carrying an automorphism with such invariants 
are given by one of the rows of Table \ref{tab}. Moreover, all cases in the table exist.
 
\begin{table}[h!]
\begin{tabular}{c|ccc|ccc|ccc|c|c}
&$N_{15}$ & $g_{15}$ & $k_{15}$ & $N_5$ &$g_{5}$ & $k_{5}$ & $N_3$ &$g_{3}$&$k_{3}$ & $d$ & {\rm NS}\\
\hline

A1&5&-&	0 &	1&2&0&2&2&0 & $(2,1,0,2)$ & $U(3)\oplus A_2\oplus A_2$\\
B1&7&-&	0 &	4&1&0&1&4&1 &$(2,0,1,4)$ & $H_5\oplus A_4$\\
B2&7&-&	0&	4&1&0&6&0&2 & $(1,2,0,6)$ & $U\oplus E_6\oplus A_2^{\oplus 3}$\\
B3&4&-&   0 &     4&1&0&0&4&0 & $(2,0,2,2)$ & $H_5\oplus A_4$\\
 D1&10&0&0&	7&1&1&6&0&2 & $(1,1,0,10)$ & $U\oplus E_6\oplus A_2^{\oplus 3}$\\
F3&9&0&	0&	10&0&1& 4&2&2 & $(1,0,2,10)$ & $H_5\oplus A_4\oplus E_8$\\
F7&12&0&0&	10&0&1& 5&2&3 & $(1,0,1,12)$ &$H_5\oplus A_4\oplus E_8$\\
F8&5&-&0&	10&0&1& 2&2&0 & $(1,0,4,6)$ & $H_5\oplus A_4\oplus E_8$\\
\end{tabular}
   \vspace{0.2cm}
\caption{Order $15$}
\label{tab}
 \end{table}

\end{theorem} 
\begin{proof}
According to \cite{ArtebaniSartiTaki}, 
the fixed locus of the cube of $\sigma_{15}$, i.e. $\sigma_5$,
is the union of a smooth curve of genus $g_5$, $k_5$ rational curves and $a_{1,5}+a_{2,5}$ isolated points, 
with $g_5,k_5,a_{1,5},a_{2,5}$ as in one of the lines of Table \ref{order5}.
\begin{table}[h!]
\begin{tabular}{c|cc|cc|c}
&$a_{1,5}$&$a_{2,5}$&$g_5$&$k_5$&$m$\\
\hline
A&1&0&2&0&5\\
B&3&1&1&0&4\\
C&3&1&-&-&4\\
D&5&2&1&1&3\\
E&5&2&0&0&3\\
F&7&3&0&1&2\\
G&9&4&0&2&1\\
\end{tabular}
\vspace{0.2cm}
\caption{Fixed locus of $\sigma_5$} \label{order5}\end{table}

We recall that  $\alpha=\sum_{C\subset\Fix(\sigma_{15})} (1-g(C))$.
In order to find all possibilities for $\Fix(\sigma_{15})$, 
we will look for a solution $a:=(a_{1,15},a_{2,15},\ldots,a_{7,15},\alpha)$ of the holomorphic Lefschetz formula 
compatible with the system of equations \eqref{sistema15}.
 
 \begin{remark}\label{rmk2}
We recall that points of type $A_{4,15}, A_{5,15}$ lie on a curve fixed by $\sigma_5$ 
and not by $\sigma_{15}$.
Thus if $a_{4,15}+a_{5,15}>0$, 
there is at least a curve in $\Fix(\sigma_{5})\backslash\Fix(\sigma_{15})$.
\end{remark}

\begin{remark} \label{genus3}
Observe that by \cite{broughton}, 
a curve of genus 3 does not admit an automorphism of order 5.
Thus if $\Fix(\sigma_3)$ contains a curve of genus $3$, such curve is also fixed by $\sigma_{15}$.
\end{remark}

We now analyze each line of the previous table separately.\\
 
\bbb{ {\bf Case A}:  corresponds to $\chi(\Fix(\sigma_5))=-1$. By equations (\ref{eqpts})
it follows that $a_{2,15}=a_{7,15}=0$. The only solution of  the holomorphic Lefschetz formula 
with this property is $a=(0,0,1,2,2,0,0,0)$. In particular $\chi(\Fix(\sigma_{15}))=5$. 
It follows from equations (\ref{sistema15}) that $d=(2,1,0,2)$. The proof thus follows as  
in the proof of Proposition \ref{prop15}.}

\bbb{{\bf Case B}: corresponds to $\chi(\Fix(\sigma_5))=4$, i.e. $\Fix(\sigma_5)$ is the disjoint 
union of a smooth curve of genus one and $4$ points. By \cite[Example 5.6]{ArtebaniSartiTaki} 
$X$ has an elliptic fibration $\pi:X\to \mathbb P^1$
which can be defined by a Weierstrass equation of the form
\[
y^2=x^3+(t^5+\alpha)x+(t^{10}+\beta t^5+\gamma),\ \alpha,\beta, \gamma \in \mathbb C,
\]
where $\sigma_5(x,y,t)=(x,y,\zeta_5 t)$. The automorphism $\sigma_5$ fixes pointwise the smooth fiber $F_0$
over $t=0$ and leaves invariant the fiber $F_{\infty}$ over $t=\infty$, which contains $4$ fixed points.
This property and the fact that $24-e(F_{\infty})$ must be divisible by $5$, imply that 
 $F_{\infty}$ is of Kodaira type $IV$, i.e. the union of three smooth rational curves intersecting 
transversally at one point. 
Observe that the elliptic fibration $\pi$ is invariant for $\sigma_3$, since the smooth fiber over $t=0$ is invariant 
for $\sigma_3$, and thus the same holds for the associated linear system.
Moreover $\sigma_3$ must preserve all fibers of $\pi$, since otherwise $15$ should divide 
$24-e(F_{\infty})=20$, a contradiction. 
The remaining singular fibers of $\pi$, considering the fact that they are preserved by $\sigma_3$ (thus $J=0$) and that $24-e(F_{\infty})=20$, are either $5$ fibers of type $IV$ or $10$ fibers of type $II$.

By the holomorphic Lefschetz formula 
and equations (\ref{eqpts}) we find that either $a=(0,1,0,0,3,0,0,0)$ or $a=(0,0,0,0,3,3,1,0)$.

If $a=(0,1,0,0,3,0,0,0)$, then  it follows from equations (\ref{sistema15}) that either 
$d=(2,0,2,2)$ or $d=(1,2,1,4)$. The first case has been considered in the proof 
of Proposition \ref{prop15} (case {\bf B3}). 
In the second case by  (\ref{sistema15}) and \cite[Table 1]{ArtebaniSarti}, $\chi_3=9$ and the fixed locus of  $\sigma_3$ contains
at least two curves. We now exclude this case.

The automorphism $\sigma_{15}$ fixes four points: three of them lie on the unique curve $F_0$ 
fixed by $\sigma_5$ and the other one is an isolated fixed point for $\sigma_5$.
By the previous description, it follows that $\sigma_3$ must fix the center of 
the fiber $F_{\infty}$ and permutes the other three fixed points of $\sigma_5$ on it 
(and thus the three components of the fiber $F_{\infty}$).
Moreover, being of types $A_{2,15}$ and $A_{5,15}$, 
the fixed points of $\sigma_{15}$ are all contained in a curve $C$ fixed by $\sigma_3$.
Since $C$ passes through the center of the fiber $F_{\infty}$,
then it is connected and by the Riemann-Hurwitz formula it is the unique fixed curve of 
$\sigma_3$ which is transversal to the fibers of $\pi$. 
On the other hand, $\sigma_3$ can not fix a curve $R$ contained in a fiber of $\pi$,
since the other singular fibers are either of type $II$, or of type $IV$,
and in both cases $R$ would intersect $C$, a contradiction. 
Thus $\sigma_3$ fixes at most one (connected) curve, so that 
the case $d=(1,2,1,4)$ is not possible.

If $a=(0,0,0,0,3,3,1,0)$, then  it follows from equations (\ref{sistema15}) that either 
$d=(2,0,1,4)$ or $d=(1,2,0,6)$. If $d=(2,0,1,4)$, then by (\ref{sistema15}) and \cite[Table 1]{ArtebaniSarti}, 
$\chi_3=-3$ and the fixed locus of $\sigma_3$ consists either of the disjoint 
union of a genus three curve and one point 
or the disjoint union of a  curve of genus four, a rational curve and one point.
The first case is not possible by Remark \ref{genus3}.

If $d=(1,2,0,6)$, then  by (\ref{sistema15}) and \cite[Table 1]{ArtebaniSarti},
$\chi_3=12$ and the fixed locus of $\sigma_3$ consists either of the union of three disjoint 
rational curves and $6$ points or the disjoint union of a curve of genus one, three rational curves and six points.
We now exclude the second case. Observe that in this case $\sigma_{15}$ fixes three points on $F_0$ 
and four isolated points in the fiber $F_{\infty}$. Six of these points are contained in a curve 
fixed by $\sigma_3$, which will intersect each fiber of $\pi$ at three points counting multiplicity.
The same argument as before shows that  $\sigma_3$ can not fix a curve contained in a fiber of $\pi$.
Thus $\sigma_3$ fixes at most three (connected) curves.

To conclude, the only possible cases have $a=(0,0,0,0,3,3,1,0)$ and either $d=(2,0,1,4)$ with $\sigma_3$ fixing a genus four curve, a rational curve and one point (case {\bf B1}), or $d=(1,2,0,6)$ with $\sigma_3$ fixing three smooth rational curves and six points (case {\bf B2}).
}

\bbb{{\bf Case C}:  in this case $\sigma_5$ fixes exactly four points, more precisely
$a_{1,5}=3$ and $a_{2,5}=1$. As before, by the holomorphic Lefschetz formula one obtains that either 
\[a=(0,1,0,0,3,0,0,0) \mbox{ or } a=(0,0,0,0,3,3,1,0).\]
In both cases $a_{4,15}+a_{5,15}>0$, 
thus this case is not possible by Remark \ref{rmk2}.}

\bbb{{\bf Case D}:
in this case
the fixed locus of $\sigma_5$ contains an elliptic curve, a smooth rational curve $R$
and $7$ isolated fixed points, with $a_{1,5}=5, a_{2,5}=2$.
The holomorphic Lefschetz formula with the restrictions of  \eqref{eqpts} gives four solutions 
for the vector $a$:
{\small
\begin{equation}
(0, 0, 0, 0, 3, 3, 1, 0),
    (0, 0, 1, 2, 2, 0, 0, 0),
    (0, 1, 0, 0, 3, 0, 0, 0),
    (3, 2, 2, 3, 0, 0, 0, 1).
\label{DE}
  \end{equation}}
  The only one compatible with equations in \eqref{sistema15} is $a=(3, 2, 2, 3, 0, 0, 0, 1)$.
By Remark \ref{rmk2} a solution with $\alpha=1$ means that only $R$ is fixed by $\sigma_{15}$.
By \eqref{sistema15}, this gives $\chi_3=9$. 
According to \cite[Table 1]{ArtebaniSarti}, there are two possibilities for $\Fix(\sigma_3)$:
\begin{enumerate}[{\bf D1}]
\item  disjoint union of 3 smooth rational curves and 6 points;
\item disjoint union of an elliptic curve, 3 smooth rational curves and 6 points.
\end{enumerate}

We now show that case D2 is not possible. 
Let 
\[
\Fix(\sigma_3)=E\cup R_1\cup R_2\cup R_3\cup\{p_1,p_2,\ldots,p_6\}
\]
and consider the elliptic fibration $\pi:X\to \mathbb P^1$ defined by the linear system $|E|$.
The automorphism $\bar\sigma_3$ induced by $\sigma_3$ on $\P^1$ 
is not the identity, since otherwise $\sigma_3$ should 
act on the general fiber of $\pi$ either as a translation (which is impossible since $\sigma_3$ is non-symplectic) 
or with fixed points (impossible, since otherwise $\sigma_3$ should fix a curve which is transverse to all fibers, 
and thus intersecting $E$).
Thus $\bar\sigma_3$ has order three and fixes two points in $\P^1$, one of them corresponding to the fiber $E$.
The smooth rational curves and the isolated points fixed by $\sigma_3$ must be components of the other invariant fiber.
This implies that such  fiber is of type $I_6^*=\tilde D_{10}$.

Since the curve $E$ is preserved by $\sigma_5$, thus the fibration $\pi$ is preserved too. 
The fixed locus of $\sigma_5$ contains a curve of genus one $E'$. 
The curve $E'$ can not be transverse to the fibers of $\pi$, since otherwise the general fiber 
of $\pi$ would have an order five automorphism with a fixed point, which is impossible by \cite[Corollary 4.7, IV]{Hartshorne}.
Thus $E'$ is one of the fibers of $\pi$.
A similar reasoning to the one used for $\sigma_3$ implies that $\sigma_5$ induces an order $5$ automorphism 
of $\P^1$, thus it preserves exactly two fibers of $\pi$. 
Observe that $\sigma_5$ must preserve both $E$, since it commutes with $\sigma_3$, 
and the fiber of type $I_6^*=\tilde D_{10}$,
since an elliptic fibration of a K3 surface can not have five fibers 
of this type (the Euler number of the fiber is $12$).
This implies that $E=E'$, thus $E$ would be a fixed curve of $\sigma_{15}$, a contradiction.}

\bbb{{\bf Case E}:
as in the previous case,
$a_{1,5}=5, a_{2,5}=2$ and the holomorphic
Lefschetz formula  with the restrictions of \eqref{eqpts} has the four solutions of \eqref{DE}.
Since in each case $a_{4,15}+a_{5,15}>0$, 
then by Remark \ref{rmk2} the only curve fixed by $\sigma_5$ is not fixed by $\sigma_{15}$ and
 $\alpha=0$.
For each one of the three possibles $a$'s with $\alpha=0$, the system \eqref{sistema15} has no solutions.
Thus there are no $\sigma_{15}$ such that $\sigma_5$ has invariants as in case E.}

\bbb{{\bf Case F}:
in this case $\Fix(\sigma_5)$ contains two rational curves $R_1,R_2$ and ten points with $a_{1,5}=7, a_{2,5}=3$. 
The holomorphic Lefschetz formula with the restrictions of \eqref{eqpts} gives nine solutions, all of them with $\alpha=0$ or $1$.
Thus at most one of the two curves $R_i$ is contained in $\Fix(\sigma_{15})$.

If $\Fix(\sigma_{15})$ contains a rational curve, then $\alpha=1$ and combining the nine solutions 
of the Lefschetz formula with \eqref{sistema15} one gets the possibilities F1-F7 of Table \ref{tab:F}.
If $\Fix(\sigma_{15})$ only contains points, then $\alpha=0$ and by \eqref{sistema15} we get possibilities F8 and F9.
 {\small \begin{table}[h!]
\begin{tabular}{c|cc|ccc|cccccccc}
&$a_{1,5}$&$a_{2,5}$&$a_{1,3}$&$g_{3}$&$k_{3}$&$a_{1,15}$&$a_{2,15}$&$a_{3,15}$&$a_{4,15}$&$a_{5,15}$&$a_{6,15}$& $a_{7,15}$&$\alpha$\\
\hline
F1&7&3& 4&0&0& 3&3&1&1&1&0&0&1 \\
F2&7&3& 4&1&1& 3&3&1&1&1&0&0&1\\
F3&7&3& 4&2&2& 3&3&1&1&1&0&0&1 \\
F4&7&3& 4&3&3& 3&3&1&1&1&0&0&1\\
F5&7&3& 5&0&1& 3&2&1&1&1&3&1&1\\
F6&7&3& 5&1&2& 3&2&1&1&1&3&1&1\\
F7&7&3& 5&2&3& 3&2&1&1&1&3&1&1 \\
F8&7&3& 2&2&0& 0&0&1&2&2&0&0&0 \\
F9&7&3& 2&3&1& 0&0&1&2&2&0&0&0 \\
\end{tabular}
\vspace{0.2cm}

\caption{Case F}
\label{tab:F}
 \end{table}
 }
 
 By Remark \ref{genus3} we exclude cases F4 and F9.
 
Case F1 has to be excluded for the following reason:
the total number of fixed points for $\sigma_{15}$ is 9 and $\sigma_{15}$ fixes a rational curve.
Thus, $a_{2,15}+a_{3,15}+a_{5,15}+a_{6,15}=5$ of the isolated fixed points for $\sigma_{15}$ lie on curves fixed by $\sigma_{3}$.
However, $\Fix(\sigma_3)$ contains just one rational curve, which is fixed by $\sigma_{15}$,
giving a contradiction.

Case F2 has to be excluded for the following reason: 
$\sigma_{15}$ acts as an automorphism of order $5$ on the elliptic curve in $\Fix(\sigma_3)$ 
and it contains fixed points, which is not possible by \cite[Corollary 4.7, IV]{Hartshorne}.
Case F6 is analogous.

In case F5, the total number of fixed points for $\sigma_{15}$ is 12:
5 of them are isolated for $\sigma_3$, thus 7 points should lie 
on the rational curve in $\Fix(\sigma_3)\backslash\Fix(\sigma_{15})$. This is not possible by the Riemann-Hurwitz formula.}

\bbb{{\bf Case G}:
in this case $\Fix(\sigma_5)$ contains three rational curves and 
all solutions of the holomorphic Lefschetz formula with the restrictions of \eqref{eqpts} have $\alpha=0$ or $1$.
Thus at most one of the three rational curves in $\Fix(\sigma_5)$ is contained in $\Fix(\sigma_{15})$.
Checking \eqref{sistema15} for all solutions in both cases $\alpha=0,1$ we find no solutions.
Thus there are no possible $\sigma_{15}$ such that $\Fix(\sigma_5)$ is as in case G.}

The N\'eron-Severi lattice of a very general K3 surface in cases A1, B1 and B3 has been given in Section \ref{sec:15}.
In the remaining cases, which have $d_{15}=1$, the rank of the N\'eron-Severi lattice in the very general case is $22-\varphi(8)=14$.
In the cases B2 and D1 we have that the rank of $S(\sigma_3)$ is $d_1+4d_5=14$.  Since $S(\sigma_3)$ is a primitive sublattice of 
${\rm NS}(X)$ by Remark \ref{ns}, we conclude that ${\rm NS}(X)=S(\sigma_{3})$. By \cite{ArtebaniSarti} the lattice $S(\sigma_{3})$ in the two 
cases is isometric to $U\oplus E_6\oplus A_2^{\oplus 3}$. A similar argument in the cases $F_3, F_7, F_8$ shows that 
for a very general $X$ we have ${\rm NS}(X)=S(\sigma_5)$ 
and that it is isometric to $H_5\oplus A_4\oplus E_8$ by \cite{ArtebaniSartiTaki}.
\end{proof}

In the following we will provide Examples for all cases collected in Table \ref{tab}, thus completing the proof of Theorem \ref{thm15}.
Examples of cases A1, B1 and B3 can be found in Section \ref{sec:15}.

\begin{example}\label{B2} (Case B2).
The elliptic K3 surface with Weierstrass equation 
\[
y^2=x^3+(t^5-1)^2
\]
has six fibers of type $IV$, over $t=\infty$ and over the zeroes of $t^5-1$.
It carries the order $15$ automorphism
\[
\sigma_{15}:(x,y,t)\mapsto(\zeta_3x,y,\zeta_{5}t).
\]
The fixed locus of $\sigma_5$ is contained in the union of the smooth fiber over $t=0$ 
and  in the fiber over $t=\infty$.
The fixed locus of $\sigma_3$ contains the section at infinity, the two sections 
defined by $x=y\pm (t^5-1)=0$ and the six centers of the fibers of type $IV$.
\end{example}

\begin{example}\label{D1} (Case D1)
This surface appears in \cite{brandhorst}. Let $X$ be the elliptic K3 surface with Weierstrass equation
\[
y^2=x^3+t^5x+1, 
\] 
The fibration has one fiber of type $III^*=\tilde E_7$ over $t=\infty$ and $15$ fibers of type $I_1$.
It carries the order $15$ automorphism
\[
\sigma_{15}:(x,y,t)\mapsto(\zeta_{15}^{10}x,y,\zeta_{15}t).
\]
The automorphism $\sigma_5=\sigma_{15}^3$ fixes the smooth fiber $E$ over $t=0$,
the smooth rational curve of multiplicity $4$ of the fiber over $t=\infty$ 
and $7$ isolated points  in the same reducible fiber. 
Thus the invariants of $\sigma_5$ are $(g_5,k_5)=(1,1)$, 
which corresponds to case D.
The elliptic curve $E$ is not fixed by $\sigma_3=\sigma_{15}^5:(x,y,t)\mapsto(\zeta_3x,y,\zeta_{3}t)$. 
The automorphism $\sigma_3$ fixes three smooth rational curves and $3$ isolated points in the fiber over $t=\infty$,
and $3$ points in the curve $E$.
\end{example}

 \begin{example}\label{F3} (Case F3) 
Let $Y$ be the double cover of $\P^2$ defined by the following equation in 
$\P(1,1,1,3)$:
\[
y^2=x_2(x_0^2x_1^3+x_2^5+x_0^5).
\]
The branch sextic $B$ is the union of a line $L$ and a quintic curve $Q$.
The surface $Y$ has four rational double points: one point of type $D_7$ at $(0,1,0,0)$ 
and three points of type $A_1$ at $(-\zeta_3^i,1,0,0)$, for $i=0,1,2$.
The minimal resolution of $Y$ is a K3 surface $X$.
The surface has the order $15$ automorphism 
\[
\sigma_{15}:(x_0,x_1,x_2,y)\mapsto(x_0, \zeta_3x_1,\zeta_5x_2,\zeta_5^3y).
\]
We will denote by $\tilde\sigma_{15}$ the lifting of $\sigma_{15}$ to $X$.
The automorphism $\sigma_3$ fixes the genus two curve $C_2$ defined by $x_1=0$ and the singular point $(0,1,0,0)$.
Thus $\tilde\sigma_3$ fixes the proper transform of $C_2$ and the union of two components and four isolated points 
in the exceptional divisor of type $D_7$. Thus we are in case F3.
\end{example}

 \begin{example}\label{F7} (Case F7) 
 Let $Y$ be the double cover of $\P^2$ defined by the following equation in 
$\P(1,1,1,3)$:
\[
y^2=x_2(x_2^5+x_1^5+x_0^3x_1x_2).
\]
The branch sextic $B$ is the union of a line $L$ and a quintic curve $Q$.
The surface $Y$ has a rational double point of type $D_{10}$ at $(1,0,0,0)$. 
The minimal resolution of $Y$ is a K3 surface $X$.
The surface has the order $15$ automorphism 
\[
\sigma_{15}:(x_0,x_1,x_2,y)\mapsto (\zeta_{5}^2x_0,\zeta_{15}^7x_1,\zeta_3^2x_2,y).
\] 
We will denote by $\tilde\sigma_{15}$ the lifting of $\sigma_{15}$ to $X$.
The automorphism $\sigma_3$ fixes the genus two curve $C_2$ defined by $x_0=0$ and the point $(1,0,0,0)$.
Thus $\tilde\sigma_3$ fixes the proper transform of $C_2$ and the union of three components and five isolated points 
in the exceptional divisor of type $D_{10}$. Thus we are in case F7.
\end{example}

 \begin{example}\label{F8} (Case F8) 
Let $Y$ be the double cover of $\P^2$ defined by the following equation in 
$\P(1,1,1,3)$:
\[
y^2=x_0^5x_1+(x_1^3-x_2^3)^2.
\]
The surface $Y$ has three rational double points of type $A_4$ at $(0,1,\zeta_3^i,0)$, with $i=0,1,2$.
The minimal resolution of $Y$ is a K3 surface $X$.
The surface has the order $15$ automorphism 
\[
\sigma_{15}:(x_0,x_1,x_2,y)\mapsto (\zeta_5x_0,x_1,\zeta_3x_2,y).
\]
We will denote by $\tilde\sigma_{15}$ the lifting of $\sigma_{15}$ to $X$.
The automorphism $\sigma_3$ fixes the genus two curve $C_2$ defined by $x_2=0$ 
and the smooth points $(0,0,1,\pm 1)$. Thus we are either in case F8 or in case A1.
The automorphism $\sigma_5$ fixes the two smooth rational curves defined by 
$x_0=y\pm (x_1^3-x_2^3)=0$ and the point $(1,0,0,0)$.
Thus its lifting $\tilde\sigma_5$ fixes two smooth rational curves, so we are in case F8.
\end{example}

\begin{remark} 
By Section \ref{moduli} the moduli space of K3 surfaces having a purely non-symplectic automorphism of order $15$ 
whose invariants are as in cases B2, D1, F3, F7  or F8  is $0$-dimensional,  since $\dim(V^\sigma)=1$. 
In cases B2 and D1 the isometry $f^*$ has order $5$ on ${\rm NS}(X)$, while in cases $F3, F7, F8$ it has order $3$.
Moreover, in all cases ${\rm rk}\,T(X)=8=\varphi(15)$. 
It follows from \cite[Theorem 5.9]{brandhorst} that there is a unique K3 surface $X$ which carries purely 
non-symplectic automorphisms of order $22$ of types B2 and D1, and a unique K3 surface carrying automorphisms
of types F3, F7, F8. Thus the K3 surfaces given in Examples \ref{B2} and \ref{D1} are isomorphic,  and the same is true for 
the K3 surfaces given in Examples \ref{F3}, \ref{F7}, \ref{F8}. 
\end{remark}


\section*{Acknowledgements}
We warmly thank  Alice Garbagnati  for her valuable suggestions and for sharing 
with us a private note on non-symplectic automorphisms of K3 surfaces. 
We also thank Antonio Laface and Alessandra Sarti for several interesting discussions.
Finally, we are grateful to the anonymous referees for pointing out a mistake (a missing component) in a previous version of the paper 
and for suggesting how to improve the presentation.

The first and last author have been partially 
supported by Proyecto FONDECYT Regular 
N. 1160897 and N. 1211708, the second author has been partially supported by  
Proyecto FONDECYT Iniciaci\'on en Investigaci\'on N. 11190428,
all authors have been supported by  
Proyecto Anillo ACT 1415 PIA CONICYT.

\bibliographystyle{plain}
\bibliography{biblio22}


\end{document}